\documentclass[a4paper,12pt]{amsart}
\usepackage{amsmath}
\usepackage{amsfonts}
\usepackage{amssymb}
\usepackage{amsthm}
\usepackage[all]{xy}
\usepackage[CJKbookmarks]{hyperref}
\usepackage{fancyhdr}
\usepackage{tikz-cd}

\newtheorem{theorem}{Theorem}[section]
\newtheorem{lemma}[theorem]{Lemma}
\newtheorem{proposition}[theorem]{Proposition}
\newtheorem{corollary}[theorem]{Corollary}

\theoremstyle{definition}
\newtheorem{definition}[theorem]{Definition}
\newtheorem{example}[theorem]{Example}

\theoremstyle{remark}
\newtheorem{remark}[theorem]{Remark}

\numberwithin{equation}{section}

\DeclareMathOperator{\Ext}{Ext}
\DeclareMathOperator{\gExt}{\underline{Ext}}

\DeclareMathOperator{\gr}{gr}

\DeclareMathOperator{\gldim}{gldim}

\DeclareMathOperator{\Gr}{Gr}
\DeclareMathOperator{\Grj}{Gr_J}

\DeclareMathOperator{\Hom}{Hom}
\DeclareMathOperator{\gHom}{\underline{Hom}}

\DeclareMathOperator{\im}{Im}
\DeclareMathOperator{\Ker}{Ker}

\DeclareMathOperator{\pdim}{pdim}

\begin{document}

\title{Generalized Koszul Algebra and Koszul Duality}


\author{Haonan Li}
\address{Shanghai Center for Mathematical Sciences, Fudan University, Shanghai 200433, China}
\email{17110840001@fudan.edu.cn}

\author{Quanshui Wu}
\address{School of Mathematical Sciences, Fudan University, Shanghai 200433, China}
\email{qswu@fudan.edu.cn}



\date{}

\dedicatory{}

\begin{abstract}
We define generalized Koszul modules and rings and develop a generalized Koszul theory for $\mathbb{N}$-graded rings with the degree zero part noetherian semiperfect. This theory specializes to the classical Koszul theory for graded rings with degree
zero part artinian semisimple developed by Beilinson-Ginzburg-Soergel and the ungraded Koszul theory for noetherian semiperfect rings developed by Green and  Martin{\'e}z-Villa.
Let $A$ be a left finite  $\mathbb{N}$-graded ring  generated in degree $1$ with $A_0$ noetherian semiperfect, $J$ be its graded Jacobson radical.
By the Koszul dual of $A$ we mean the Yoneda Ext ring $\gExt_A^\bullet(A/J,A/J)$.
If $A$ is a generalized Koszul ring and $M$ is a generalized Koszul module, then it is proved that the Koszul dual of the Koszul dual of $A$ is the associated graded ring $\Grj A$ and
the Koszul dual of the Koszul dual of $M$ is the associated graded module $\Grj M$.
If $A$ is a locally finite algebra, then
the following statements are proved to be equivalent: $A$ is generalized Koszul; the Koszul dual $\gExt_A^\bullet(A/J,A/J)$ of $A$ is (classically) Koszul;   $\Grj A$ is (classically) Koszul;
the opposite ring $A^{op}$ of $A$ is generalized Koszul.
As an application, it is proved that if $A$ is generalized Koszul with finite global dimension then $A$ is generalized AS regular if and only if the Koszul dual of $A$ is self-injective.
\end{abstract}

\maketitle

\section{Introduction}
Koszul rings and Koszul duality play an important role in commutative algebra, algebraic
topology, (noncommutative) algebraic geometry, and in the theory of quantum groups, representation theory of algebras  and  Lie theory (see \cite{Pri, Lof, BGS1, Fr, CPS, Sm}, etc.).

Classically, a Koszul ring $A$ means an $\mathbb{N}$-graded ring $A=A_0\oplus A_1\oplus A_2\oplus \cdots$ with $A_0$ artinian semisimple and $A_0 \cong A/A_{\geqslant 1}$  having a linear projective resolution as a graded $A$-module (see \cite[Definition 1.2.1]{BGS1}).
Here are some major results in the classical Koszul theory. Suppose that $A$ is left finite, that is, every $A_n$ is finitely generated as an $A_0$-module, and $A_0$ is artinian semisimple. Then the following are equivalent: $A$ is Koszul; $A^{op}$ is Koszul \cite[Theorem 2.2.1]{BGS1}; the Koszul dual $\gExt^\bullet_A(A_0,A_0)$ (or the Yoneda Ext ring) of $A$ is Koszul \cite[Theorem 2.10.2]{BGS1} (see also \cite[Theorem 6.1]{GM1}); $\gExt^n_A(A_0,A_0)$ is concentrated in degree $-n$ for all $n \geqslant 0$ \cite[Proposition 2.1.3]{BGS1}; $A$ is isomorphic to the Koszul dual of the Koszul dual of $A$ \cite[Theorem 2.4]{GM2} (see also \cite[Theorems 10.1, 10.2]{GM1}, \cite[Theorem 2.10.2]{BGS1}). Moreover, there is a duality between the category of Koszul modules over a Koszul ring $A$ and the category of Koszul modules over the Koszul dual of $A$ \cite[Theorem 5.2]{GM2}.

The Koszul rings in the sense of \cite{BGS1} will be called sometimes classically Koszul later in  this paper.

There are several generalized Koszul theories in literature, where $A_0$ is not assumed being artinian semisimple.

To develop a unified approach to Koszul duality and cotilting theory, $T$-Koszul algebras were first defined in \cite{GRS} for left finite graded algebras $A$ with $A_0$ artinian, where $T$ is a Wakamatsu cotilting $A_0$-module. For locally finite graded algebra $A$ such that  $A_0$ has finite global dimension, an equivalent definition of $T$-Koszul algebras was given in \cite[Definition 4.1.1]{Ma}.
Similar to the classical Koszul theory, the following statements were proved: $A^{op}$ is $T^*$-Koszul if $A$ is $T$-Koszul \cite[Theorem 4.1.2]{Ma}; each $T$-Koszul algebra has a $T^*$-Koszul dual algebra which is the Yoneda Ext algebra of $T$ \cite[Theorem 4.2.1(a)]{Ma}; the $T^*$-Koszul dual of the $T$-Koszul dual algebra is isomorphic to the original algebra \cite[Theorem 4.2.1(b)]{Ma}; and there is a duality between categories of $T$-Koszul modules over a $T$-Koszul algebra  and its $T$-Koszul dual algebra \cite[Theorem 4.3.1]{Ma}.
If $A_0$ is artinian semisimple and $T=A_0$, then $T$-Koszul algebras are classically Koszul. $T$-Koszul theory specializes to classical Koszul
theory and to Wakamatsu tilting theory (see \cite{Ma} and the references therein).

An earlier generalization of a graded ring
$A$ being Koszul was given by Woodcock in \cite{Wo}, where $A$ is
both a left and a right projective $A_0$-module, but with no restrictions on $A_0$.
 This is not necessarily the case in the setting of \cite{GRS}.
 The point of view in \cite{Wo} is considering the Koszul complex
as the defining property of the so called Koszul modules.
However, assuming that $A$ is a
left finite graded ring generated in degree $1$ with $A_0$ artinian, and choosing $T=(A_0)^*$, the vector space dual of $A$, one can see that $A$ is $T$-Koszul in the sense of \cite{GRS}.

 To apply Koszul theory to study the Ext groups of representations of finite EI categories,  Li developed a generalized Koszul theory for locally finite $\mathbb{N}$-graded $k$-algebra $A$ with $A_0$ self-injective instead of semisimple \cite{Li1}.  The results were proved to be true more generally for $A_0$ with finitistic dimension $0$ \cite{Li2}. In fact, both finite dimensional local algebras and self-injective algebras have finitistic dimension $0$. In \cite{Li1, Li2}, generalized Koszul modules and algebras were defined via linear projective resolutions as in the classical case.
 The following results were proved in \cite[Section 1]{Li2}: if further $A$ is projective as an $A_0$-module then a graded $A$-module $M$ is generalized Koszul if and only if it is a projective $A_0$-module and  $\gExt_A^1(A_0,A_0)\cdot \gExt_A^{n-1}(M,A_0)=\gExt^{n}_A(M,A_0)$ for all $n > 0$; if $A$ is generalized Koszul then there is a duality between the generalized Koszul modules over $A$ and over $\gExt_A^{\bullet}(A_0,A_0)$ via $M \mapsto \gExt_A^{\bullet}(M,A_0)$; $A$ is generalized Koszul if and only if $A$ is  $A_0$-projective and $A/\mathfrak{R}$ is classically Koszul where $\mathfrak{R}=AJ(A_0)A$; if $A$ is $A_0$-projective, then a graded $A$-module $M$ is generalized Koszul if and only if it is  $A_0$-projective and  $M/{\mathfrak{R}M}$ is a classically Koszul $A/\mathfrak{R}$-module.

The concept of Koszul algebras has been generalized  to higher Koszul algebras \cite{Ber, GMMZ, HY, L3, LHL} and to  categories \cite{Man, BGS2, MOS}, which are not the topics in this paper.

A linear projective resolution of a graded module $M$ can be characterized by the property of (the radical filtration of) syzygies of $M$ \cite[Proposition 3.1 and Lemma 5.1]{GM1} (see Proposition \ref{qk and classical koszul}). Using this characterization,  Green and Martin{\'e}z-Villa defined (strongly) quasi-Koszul modules and rings, and developed a Koszul theory for noetherian semiperfect rings \cite{GM1,GM2}.

In this paper, we develop a generalized Koszul theory for $\mathbb{N}$-graded rings $A$ with the degree $0$ part $A_0$ noetherian semiperfect (not necessarily semisimple), which specializes to the classical graded Koszul theory if $A_0$ is artinian semisimple and the ungraded Koszul theory for noetherian semiperfect ring if $A$ is concentrated in degree $0$.
More explicitly, let $A$ be an $\mathbb{N}$-graded ring with $A_0$ noetherian semiperfect, $J=J(A_0)\oplus A_{\geqslant 1}$ be the graded Jacobson radical of $A$ and $S=A/J$.
Let $M$ be a left finite and bounded below graded $A$-module. It follows from Proposition  \ref{minimal projective resolution} that $M$ has a minimal graded projective resolution.
Following the idea in \cite{GM1}, a graded $A$-module $M$ with a minimal graded projective resolution $P_\bullet \to M \to 0$  is called Koszul (resp. quasi-Koszul) if $J^k\Ker d_n = \Ker d_n\cap J^{k+1}P_n$ for all $n,k\geqslant 0$ (resp.  $J\Ker d_n = \Ker d_n\cap J^2P_n$ for all $n$) (see Definition \ref{our-definition}).
Then, a left finite graded ring $A$ generated in degree $1$ with $A_0$ noetherian semiperfect is called a (quasi-)Koszul ring if $S \cong A_0/{J(A_0)}$ is  a  (quasi-)Koszul $A$-module. Here, we rename ``strongly quasi-Koszul" in \cite{GM1} by ``Koszul" but remain the name ``quasi-Koszul".

Let $E(A)=\gExt_A^{\bullet}(A/J,A/J)$ be the Yoneda Ext ring  of $A$.
The first result says that $A$ is a quasi-Koszul ring if and only if $E(A)$ is generated in degree $1$; and if $A$ is quasi-Koszul then $M$ is quasi-Koszul if and only if $\gExt_A^{\bullet}(M,A/J)$ is generated
in degree $0$ as an $E(A)$-module (see Theorem \ref{qK iff E(M) generated in degree 0}). Sometimes, $E(A)$ is called the Koszul dual of $A$.
\begin{theorem} \label{theorem 1}
\begin{itemize}
\item[(1)] $M$ is quasi-Koszul if and only if, for any $n>0$, $$\gExt_A^1(A/J,A/J)\cdot \gExt_A^{n-1}(M,A/J)=\gExt^n_A(M, A/J).$$
\item[(2)] $A$ is a quasi-Koszul ring if and only if $\gExt_A^{\bullet}(A/J,A/J)$ is generated by  $\gExt^1_A(A/J,A/J)$ over $\gHom_A(A/J,A/J)$.
\end{itemize}
\end{theorem}

We describe in Example \ref{qk-but-not-sqk} a quasi-Koszul ring which is not a  Koszul ring.

Koszul rings and modules are characterized in the following by using Theorem \ref{theorem 1} (see Theorem \ref{A sqk implies E(A) Koszul}).
\begin{theorem}\label{theorem 2}
\begin{itemize}
\item[(1)] $A$ is a Koszul ring if and only if $\gExt_A^\bullet(A/J,A/J)$ is a classical Koszul ring.
\item[(2)] Suppose that $A$ is Koszul. Then $M$ is a Koszul $A$-module if and only if $\gExt_A^{\bullet}(M,A/J)$ is a classical Koszul $E(A)$-module.
\end{itemize}
\end{theorem}

Let $\Grj A=\oplus (J^i/J^{i+1})$, and $\Grj M=\oplus (J^iM/J^{i+1}M)$ for any graded $A$-module $M$. Then $\Grj M$ is a graded $\Grj A$-module. If $A_0$ is semisimple then $\Grj A\cong A$ and $\Grj M\cong M$.
We emphasize that $\Grj A$ is viewed  as a graded ring  via the graded degree induced by $J$-adic filtration, and $E(A)$ is viewed  as a graded ring  via the homological degree.

Let $\mathcal{K}_A$,  $\mathcal{K}_{E(A)}$ and $\mathcal{K}_{\Grj A}$ be the full subcategories  of finitely generated Koszul $A$-modules, Koszul ${E(A)}$-modules and  Koszul ${\Grj A}$-modules  respectively in the corresponding categories.
 Let
 $$\mathcal{E}=\gExt_A^\bullet(-,A/J),
\mathcal{F}=\gExt_{E(A)}^\bullet(-,E(A)/J_{E(A)}),
\mathcal{G}= \gExt^\bullet_{\Grj A}(-,A/J).$$

The following is a generalized version of Koszul algebra duality and Koszul module duality, which generalizes \cite[Theorem 2.10.2]{BGS1}, \cite[Theorem 6.3 (4)]{Sm}  and \cite[Theorems 10.1, 10.4]{GM1}.

\begin{theorem}\label{theorem 3}
Let $A$ be a Koszul ring. Then
\begin{itemize}
\item[(1)] $E(E(A))\cong \Grj A$ as graded rings.
\item[(2)]
The functors $\mathcal{E}, \mathcal{F}$ and $\mathcal{G}$ restrict to
$$\xymatrix{
\mathcal{K}_A\ar[r]^{\mathcal{E}} &\mathcal{K}_{E(A)}\ar@<1mm>[r]^{\mathcal{F}}& \mathcal{K}_{\Grj A}\ar@<1mm>[l]^{\mathcal{G}}.
}$$
For any $M\in \mathcal{K}_{A}$, $\mathcal{F}\mathcal{E}(M)\cong \Grj M$ as graded $\Grj A$-modules.
\item[(3)] The functors $\mathcal{F}$ and $\mathcal{G}$  give a duality between $\mathcal{K}_{E(A)}$ and $\mathcal{K}_{\Grj A}$.
\end{itemize}
\end{theorem}


 The following characterization for Koszul rings is proved under the assumption that  $A_0$ is artinian. Similar result holds for Koszul modules (see Theorem \ref{main-resut}).

\begin{theorem} \label{theorem 4}
Suppose that $A$ is a left finite $\mathbb{N}$-graded ring generated in degree $1$  with $A_0$ artinian. Then the following are equivalent.
\begin{itemize}
\item[(1)] $A$ is a  Koszul ring.
\item[(2)] $E(A)$ is a classical Koszul ring.
\item[(3)] $\Grj A$ is a classical Koszul ring.
\item[(4)] $E(A)$ is generated in degree $1$, and
$E(E(A))\cong \Grj A$ as graded rings.
\end{itemize}
Moreover if $A$ is right finite then the above statements are also equivalent to
\begin{itemize}
\item[(5)] $A^{op}$ is a  Koszul ring.
\end{itemize}

\end{theorem}

In the representation theory of quasi-hereditary algebras $A$, people are interested in the Koszul property of $\Grj A$ because of its connection with Kazhdan-Lusztig theory. Theorem \ref{theorem 4} gives a characterization of $A$ when $\Grj A$ is classically Koszul. In some sense Theorem \ref{theorem 4} answers the Questions in \cite{CPS} which concern the Koszulity of $A$ and $\Grj A$ (see subsection \ref{questions-in-CPS}).

For locally finite $\mathbb{N}$-graded algebras (with degree zero part not necessarily semisimple), generalized Artin-Schelter (for short, AS) regular property was studied  in \cite{MV2, MS, MM, RR1} (see Definition \ref{definition of GAS regluar} and Theorem \ref{equivalent definitions of AS regular}). Generalized AS-regular algebras are closely related to twisted Calabi-Yau algebras (see \cite[Thoeorem 1.5]{RR1}). It is well known that a connected Koszul algebra of finite global dimension is AS regular if and only if its Yoneda Ext algebra is Frobenius \cite[Proposition 5.10]{Sm}. The final result in this paper
is to prove that this fact holds for basic locally finite Koszul algebras  of finite global dimension. The proof is reduced to the classically Koszul case via $\Grj A$, which was proved essentially in \cite[Theorem 5.1]{MV1} (see Theorems \ref{A is generalized AS regular iff GrA is} and \ref{Char-generalized Koszul AS-regular algebra}). Recall that an $\mathbb{N}$-graded $k$-algebra $A$ is called basic if the degree $0$ part of $A/J$ is a finite direct sum of $k$. 

\begin{theorem}\label{theorem 5}
Suppose that $A$ is a basic locally finite Koszul algebra of finite global dimension. Let $E(A)=\gExt_A^\bullet(A/J,A/J)$ be the Yoneda Ext algebra of $A$. The following are equivalent.
\begin{itemize}
\item[(1)] $A$ is generalized AS regular.
\item[(2)] $\Gr_J A$ is generalized AS regular.
\item[(3)] $E(A)$ is a self-injective algebra.
\end{itemize}
\end{theorem}

The paper is organized as follows. In section 2, we introduce Yoneda products and minimal graded projective resolutions, whose existence is given for bounded below modules over some  $\mathbb{N}$-graded rings.  In section 3, we define (quasi-)Koszul modules and rings following the ideas in \cite{GM1}, and prove Theorem \ref{theorem 1}, which is about the  generating property of the Yoneda Ext rings  (resp. modules) of quasi-Koszul rings (resp. modules). In section 4, we prove Theorem \ref{theorem 2} and Theorem \ref{theorem 3}. In section 5, we consider the associated graded rings (resp. modules) of Koszul rings (resp. modules) with respect to the $J$-adic filtration and prove Theorem \ref{theorem 4}. In the last section, we focus on the locally finite algebras, and prove that if $A$ is Koszul then $A$ is generalized AS regular if and only if so is $\Grj A$ and if and only if $E(A)$ is a self-injective.

\section{Preliminaries}
Yoneda products play a key role in this paper and are repeatedly used. We recall the definition of Yoneda products first.

Let $R$ be a ring and $R^{op}$ be its opposite ring. When we say a module, it always means a left module. An $R^{op}$-module is exactly a right $R$-module.

\subsection{Yoneda products}\label{Yoneda Products}
Let $X$, $Y$ and $Z$ be $R$-modules. The map $\Ext_R^j(Y,Z) \times \Ext_R^i(X,Y)\to \Ext_R^{i+j}(X,Z)$  defined in the following is called the {\it Yoneda product} of the Ext groups.

Let $P_\bullet \to X \to 0$ and $Q_\bullet\to Y \to 0$ be projective resolutions. For any $\bar{\alpha}\in \Ext_R^i(X,Y)$ and $\bar{\beta}\in \Ext_R^j(Y,Z)$  represented by $\alpha\in \Hom_R(P_i, Y)$ and $\beta\in \Hom_R(Q_j, Z)$ respectively, there is a commutative diagram
\begin{center}
\begin{tikzcd}
P_{i+j} \arrow[r] \arrow[d, "\alpha^j"] & \cdots \arrow[r] & P_i \arrow[r] \arrow[d, "\alpha^0"] \arrow[rd, "\alpha"] & \cdots \arrow[r] & X \arrow[r] & 0 \\
Q_j \arrow[r] \arrow[d, "\beta"]        & \cdots \arrow[r] & Q_0 \arrow[r]                                            & Y \arrow[r]      & 0           &   \\
Z                                       &                  &                                                          &                  &             &
\end{tikzcd}
\end{center}
where $\alpha^0,\cdots, \alpha^j$ are the lifting of $\alpha$. Then
$$\bar{\beta}\cdot \bar{\alpha}: =\overline{\beta\circ \alpha^j}\in \Ext_R^{i+j}(X, Z)$$
is well-defined, and it is called the Yoneda product of $\bar{\alpha}$ and $\bar{\beta}$ (see, for example, \cite[Chapter 3]{ML}).

Let $\Ext_R^\bullet(X,Y)=\mathop{\oplus}\limits_{i\geqslant 0}\Ext^i_R(X,Y)$.  Then, with the Yoneda product, $\Ext_R^\bullet(Y,Y)$ is a graded ring, and $\Ext_R^\bullet(X,Y)$ is a graded $\Ext_R^\bullet(Y,Y)$-module. 

The graded ring $\Ext_R^\bullet(S,S)$ is called the {\it Yoneda Ext ring} of $R$ where $S=R/J_R$ with $J_R$ the Jacobson radical of $R$.

Yoneda products can be defined similarly in graded module categories.

\subsection{Minimal (graded) projective resolution}
Recall that a ring $R$ is semiperfect (resp. left perfect) if $R/J(R)$ is artinian semisimple and any idempotent of $R/J(R)$ can be lifted to $R$  (resp. $J(R)$ is left $T$-nilpotent) where $J(R)$ is the Jacobson radical of $R$.  A projective cover of an $R$-module $M$ is a surjective morphism $\pi:P\to M$ where $P$ is an $R$-projective module and its kernel is a superfluous submodule.
Projective cover of a module is unique up to isomorphism if it exists.

 Perfect rings and semiperfect rings are characterized by the existence of projective covers (see, for instance, \cite[Theorem 24.16]{L}).

\begin{proposition}
For any ring $R$, the following are equivalent.
\begin{itemize}
\item[(1)] $R$ is semiperfect (resp. left perfect).
\item[(2)] Every finitely generated left $R$-module (resp. Every left $R$-module) has a projective cover.
\end{itemize}
\end{proposition}

We are working on graded rings and modules.
For basic definitions and facts concerning  graded rings and  filtered rings refer to \cite{LO} or \cite{NO} for references.

Let $A$ be an $\mathbb{N}$-graded ring, $J=J(A_0)\oplus A_{\geqslant 1}$ be the graded Jacobson radical of $A$ and $S=A/J$. Let $M$ be a graded $A$-module. For any integer $l$, the $l$-shift $M(l)$ of $M$ is a graded module with $M(l)=M$ as ungraded modules but with the grading $M(l)_i=M_{i+l}$.

A graded projective resolution of a graded $A$-module  $M$
\begin{equation}\label{proj-reso}
\cdots \to P_{i+1}\xrightarrow{d_{i+1}} P_i\xrightarrow{d_i}\cdots \to P_0\xrightarrow{d_0} M\to 0
\end{equation}
 is called {\it minimal} if $\im d_{i+1}\subseteq JP_i$ for all $i \geqslant 0$. If each $P_i$ in the projective resolution \eqref{proj-reso} is generated in degree $i$, then we say $M$ has a {\it linear projective resolution} or  \eqref{proj-reso} is a  linear projective resolution of $M$.

Any linear projective resolution is minimal.

For any nonzero bounded below graded $A$-module $M$, let $l(M)$ to be the minimal integer $i$ such that $M_{i}\neq 0$.

The following is a more general version of Nakayama's lemma.
\begin{lemma}\label{Nak lemma}
Let $A$ be an $\mathbb{N}$-graded ring.
\begin{itemize}
\item[(1)] If $M$ is a  nonzero bounded below graded $A$-module with $M_{l(M)}$ finitely generated as an $A_0$-module, then $JM \lneq M$.
\item[(2)] $J(A_0)$ is left $T$-nilpotent  if and only if  $JM \lneq M$ for any nonzero bounded below graded $A$-module $M$.
\end{itemize}
\end{lemma}

\begin{proof} (1) If
 $M=JM=J(A_0)M+A_{\geqslant 1}M$, then $M_i\subseteq J(A_0)M_i$ when $i=l(M)$, and so $M_i = J(A_0)M_i$, which implies that $M_i=0$.

 (2) By \cite[Lemma 28.3]{AF}, $J(A_0)$ is left $T$-nilpotent  if and only if  $JM_0 \lneq M_0$ for any nonzero $A_0$-module $M_0$.
\end{proof}

Whenever we say that Nakayama's lemma holds for $M$ we mean $JM \lneq M$.
If $P_\bullet\to M \to 0$ is a minimal graded projective resolution of $M$, and Nakayama's lemma holds for the quotients of all $P_i$, then $P_i$ is a graded projective cover of $\im d_{i}$, that is, $\ker d_i$ is a graded superfluous submodule of $P_i$. In particular, $P_0$ is a graded projective cover of $M$.

An $\mathbb{N}$-graded ring $A$ is called {\it left (resp. right) finite} if $A_i$ is a finitely generated left (resp. right) $A_0$-module for all $i$. A graded left (resp. right) $A$-module $M$ is called {\it left (resp. right) finite} if $M_i$ is a finitely generated left (resp. right) $A_0$-module for all $i$.

In this paper, a noetherian or artinian ring means a left noetherian or left artinian ring.

\begin{proposition}\label{minimal projective resolution}
Let $A$ be an $\mathbb{N}$-graded ring, $M$ be a bounded below graded $A$-module. \begin{itemize}
\item[(1)] If $A$ is left finite with $A_0$ noetherian semiperfect, and $M$ is left finite, then $M$ has a minimal graded projective resolution.
\item[(2)] If $A_0$ is left perfect, then $M$ has a minimal graded projective resolution.
\end{itemize}
\end{proposition}

\begin{proof} (1)
Since $A_0$ is semiperfect, it has a finite complete set of orthogonal primitive idempotents $\{e_i\}_{i=1}^n$ which is also a finite complete set of orthogonal primitive idempotents of $A$. Each indecomposable $A_0$-projective module ${A_0e_i}$ can be viewed as a graded quotient $Ae_i/{A_{\geqslant 1} e_i}$   of the indecomposable projective $A$-module $Ae_i$.

Let $\overline{M}=M/JM=\oplus \overline{M}_i$. Then, for any $i$, $\overline{M}_i$ is a finitely generated $A_0$-module. Let $P_i'$ be aprojective cover of $A_0$-module $\overline{M}_i$.  By lifting $P_i'$ to a projective $A$-module $P_i$ via the $(Ae_j)'s$, we have a natural graded $A$-module morphism $\pi_i$ which is the composition of $P_i\to P_i'\to \overline{M}_i$. By a degree shifting, we may assume that $P_i$ is a graded projective cover of $A$-module $\overline{M}_i$ with $l(P_i) \geqslant i$ by Nakayama's lemma.

Since  $P_i$ is finitely generated, it is left finite. It follows from $l(P_i) \geqslant i$ that $P=\oplus_i P_i$ is left finite. Then $P$ is a graded projective cover of $\overline{M}$ with the surjective morphism $\pi=\oplus \pi_i$. Since $P$ is graded projective, there is a graded morphism $\pi': P \to M$ so that the following diagram is commutative, that is, $\epsilon\circ\pi'=\pi$.
\begin{center}
\begin{tikzcd}
                        & P \arrow[ld, "\pi'"', dashed] \arrow[d, "\pi"] &   \\
M \arrow[r, "\epsilon"] & \overline{M} \arrow[r]                 & 0
\end{tikzcd}
\end{center}
 Then $\pi'(P)+JM=M$. By Nakayama's lemma, $\pi'$ is surjective. Hence $P$ is a graded projective cover of $M$ because $\Ker \pi'\subseteq \Ker \pi\subseteq JP$.

 Since $A_0$ is noetherian, $\Ker \pi'$ is also left finite. Repeating the process by replacing $M$ with $\Ker \pi'$, we may construct a minimal graded projective resolution of $M$ inductively.

 (2) By a similar proof of (1) and  Lemma \ref{Nak lemma} (2).
\end{proof}

It follows from the proof of Proposition \ref{minimal projective resolution} that any left finite bounded below projective $A$-module is isomorphic to a  direct sum of the shifts of the indecomposable projective $A$-modules $Ae_i$ if $A_0$ is semiperfect.

\section{Koszul rings and modules}
In this section we first recall the definitions of classical graded Koszul rings, Koszul modules  \cite{BGS1}, and their characterizations by the property of syzygies given in \cite{GM1}.
By using the characterizations in graded case, Green and Martin{\'e}z-Villa defined (strongly) quasi-Koszul modules and rings in \cite{GM1} for ungraded noetherian semiperfect rings.
Following the idea in \cite{GM1} we define (quasi-)Koszul rings and modules in a more general setting. Theorem \ref{qK iff E(M) generated in degree 0} is the main result in this section, which characterizes the  quasi-Koszul property via the generating property of the Yoneda Ext ring and module, and hence generalizes the results in both classical graded and ungraded cases.

\subsection{Definition of (quasi-)Koszul modules and rings}

\begin{definition}\cite[Definitions 1.2.1, 2.14.1]{BGS1}
Let $A$ be an $\mathbb{N}$-graded ring with $A_0$ artinian semisimple. Suppose $M$ is a graded $A$-module generated in degree $0$. If $M$ has a linear projective resolution, then $M$ is called {\it classically Koszul}.
If $A_0$ considered as a graded $A$-module is a classically Koszul module, then  $A$ is called a {\it classically Koszul ring}.
\end{definition}

Recall that an $\mathbb{N}$-graded ring $A=\oplus A_{i \geqslant 0}$ is called {\it generated in degree $1$} if $A_iA_j=A_{i+j}$ for all $i, j \geqslant 0$ \cite[Lemma 2.1]{GM1}. In this case, $A$ is also said to be generated  by $A_1$ over $A_0$  \cite[Definition 1.2.2]{BGS1}.
An graded $k$-algebra $A$ is called {\it locally finite} if $A_i$ is finite dimensional as a $k$-vector space for all $i\geqslant 0$.
An $\mathbb{N}$-graded locally finite $k$-algebra generated in degree $1$ with $A_0$ being a direct sum of $k$ is called a {\it graded quiver algebra}, as it can be viewed as a quotient of a path algebra $kQ$ for a finite quiver $Q$ \cite[Proposition 1.1.1]{MV1}. In fact, most results proved in \cite{GM1, GM2} for graded quiver algebras hold for left finite $\mathbb{N}$-graded ring with $A_0$ artinian semisimple.

Here is a characterization of classical Koszul modules proved in \cite[Proposition 3.1 and Lemma 5.1]{GM1}.
\begin{proposition}\label{qk and classical koszul}
Let $A$ be an 
$\mathbb{N}$-graded ring generated in degree $1$ with $A_0$ artinian semisimple.
Suppose that $M$ is a graded $A$-module generated in degree $0$, and
$\cdots \to P_n\xrightarrow{d_n} P_{n-1}\to\cdots\to P_0\xrightarrow{d_0} M\to 0$
is a minimal graded projective resolution of $M$. Then the following are equivalent:
\begin{itemize}
\item[(1)] $M$ is a classical Koszul module.
\item[(2)] For any $n\geqslant 0$, $J\Ker d_n = \Ker d_n\cap J^2P_n$.
\item[(3)] For any $n,k\geqslant 0$, $J^k\Ker d_n = \Ker d_n\cap J^{k+1}P_n$.
\end{itemize}
\end{proposition}

The property (3) can be interpreted as follows: the $1$-shifting of the $J$-adic filtration on the syzygies is exactly the submodule filtration induced from the $J$-adic filtration of the projective modules in the minimal projective resolution.

Quasi-Koszul rings and modules were defined in \cite{GM1} first by the generating property of Koszul dual and were characterized by the condition (2) in Proposition \ref{qk and classical koszul} (see Theorem \ref{Theorem 4.4 in {GM1}}).

\begin{definition}\cite[page 263]{GM1}\label{ungraded qK ring}
Let $R$ be a noetherian semiperfect ring, $J$ be its Jacobson radical and $S=R/J$.
\begin{itemize}
\item[(1)]
A finitely generated $R$-module $M$ is called {\it quasi-Koszul} if, for any $n>0$,
$\Ext_R^1(S,S)\cdot \Ext_R^{n-1}(M,S)=\Ext^n_R(M,S)$.
\item[(2)] The ring $R$ is called {\it quasi-Koszul} if $S$ is a quasi-Koszul module, that is, $\Ext_R^{\bullet}(S,S)$ is generated by  $\Ext^1_A(S,S)$ over $\Hom_A(S,S)$.
\end{itemize}
\end{definition}

This means that $R$ is quasi-Koszul if and only if the Yoneda Ext ring $\Ext_R^{\bullet}(S,S)$ of $R$ is generated in degree $1$, and if $R$ is quasi-Koszul, then $M$ is quasi-Koszul if and only if $\Ext_R^{\bullet}(M,S)$ is generated in degree $0$ as a graded $\Ext_R^{\bullet}(S,S)$-module.
\begin{theorem} \cite[Theorem 4.4]{GM1} \label{Theorem 4.4 in {GM1}} Let $R$ be a noetherian semiperfect ring, $J$ be its Jacobson radical and $S=R/J$.
\begin{itemize}
\item[(1)] A finitely generated $R$-module $M$ is  quasi-Koszul if and only if $M$ has a minimal projective resolution $P_\bullet \to M \to 0$ such that $J\Ker d_n = \Ker d_n\cap J^2P_n$ for all $n$.
\item[(2)] The ring $R$ is  quasi-Koszul if and only if ${}_RS$ has a minimal projective resolution $P_\bullet \to S \to 0$ such that $J\Ker d_n = \Ker d_n\cap J^2P_n$ for all $n$.
\end{itemize}
\end{theorem}

Strongly Quasi-Koszul rings and modules were defined  in \cite{GM1} by the condition (3) in Proposition \ref{qk and classical koszul}.

\begin{definition}\cite{GM1}
Let $R$ be a noetherian semiperfect ring and $J$ be its Jacobson radical.
\begin{itemize}
\item[(1)]
A finitely generated
$R$-module $M$ is  called {\it strongly quasi-Koszul} if $J^k\Ker d_n = \Ker d_n\cap J^{k+1}P_n$ for any $n,k\geqslant 0$, where $P_\bullet \to M \to 0$ is
a minimal projective resolution of $M$.
\item[(2)] $R$ is called a {\it strongly quasi-Koszul ring} if $R/J$ is a strongly quasi-Koszul $R$-module.
\end{itemize}
\end{definition}

The Auslander algebra of a finite-dimensional algebra of finite type over an algebraically closed field is a quasi-Koszul algebra \cite[Theorem 9.6]{GM1}. If $R$ is further assumed  being hereditary, then any quasi-Koszul $R$-module is strongly quasi-Koszul \cite[Lemma 5.1(b)]{GM1}. 
A ring which is quasi-Koszul but not strongly quasi-Koszul is given in Example \ref{qk-but-not-sqk}.

 To unify the notion of the classical graded Koszulity and the ungraded (strongly) quasi-Koszulity, we propose a definition of (quasi) Koszulity as follows.
\begin{definition} \label{our-definition}
Let $A$ be a left finite $\mathbb{N}$-graded ring generated in degree $1$  with $A_0$ noetherian semiperfect, $J$ be its graded Jacobson radical.
Let $M$ be a graded $A$-module with a minimal graded projective resolution
$$\cdots\to P_n\xrightarrow{d_n} P_{n-1}\to \cdots P_0 \xrightarrow{d_0} M\to 0.$$
\begin{itemize}
\item[(1)] If $J\Ker d_n = \Ker d_n\cap J^2P_n$ for all $n$, then $M$ is called {\it quasi-Koszul}.
\item[(2)] If  $J^k\Ker d_n = \Ker d_n\cap J^{k+1}P_n$ for any $n,k\geqslant 0$, then $M$ is called {\it Koszul}.
\end{itemize}

If $S=A/J$ is (quasi-)Koszul as a graded $A$-module, then $A$ is called a {\it (quasi-)Koszul ring}.
\end{definition}

It follows from the definition that any finite direct sum and any direct summand of (quasi-)Koszul modules are  (quasi-)Koszul.  The  shift of (quasi-)Koszul modules is also (quasi-)Koszul, unlike the classical graded Koszulity  defined by linear projective resolutions.

\begin{remark}
Suppose that $M$ is a left finite bounded below graded $A$-module. If $M$ is (quasi-)Koszul then $M_{l(M)}$ is (quasi-)Koszul as an $A_0$-module. In particular, if $A$ is (quasi-)Koszul then $A_0$ is (quasi-)Koszul. In fact,
without loss of generality, we may assume that $l(M)=0$. Suppose
$$\cdots\to P_n\xrightarrow{d_n} P_{n-1}\to \cdots \to P_0 \xrightarrow{d_0} M\to 0,$$
is a minimal graded projective resolution of $M$. Then it is clear that
$$\cdots\to (P_n)_0\xrightarrow{} (P_{n-1})_0\to \cdots \to (P_0)_0\xrightarrow{} M_0\to 0$$
is a minimal projective resolution of $M_0$ as an $A_0$-module.
\end{remark}

\begin{remark}
Note that if $A$ is a classical graded Koszul ring, then $A$ is generated in degree $1$ by \cite[Proposition 1.2.3]{BGS1}. So for a (quasi-) Koszul ring $A$ and a graded $A$-module $M$, if $A_0$ is artinian semisimple and $M$ is generated in degree $0$, our definition of (quasi-)Koszul modules coincides with the definition of the classical graded Koszul module by Proposition \ref{qk and classical koszul}; if $A$ is concentrated in degree $0$ and $M$ is finitely generated concentrated in degree $0$, then our definition  of (quasi-)Koszul modules coincides with the definition in the ungraded case (see Definition \ref{ungraded qK ring} and Theorem \ref{Theorem 4.4 in {GM1}}). Therefore Definition \ref{our-definition} unifies and generalizes the concepts of classical graded Koszul rings and ungraded (strongly) quasi-Koszul rings.

Although the finiteness condition is not necessary in the definition of classical Koszul rings and modules,  it is needed to have the duality, as remarked before \cite[Definition 1.2.4]{BGS1} (see also Theorem \ref{qK iff E(M) generated in degree 0}). In \cite{GM1,GM2}, the algebras considered are the graded quiver algebras, so the finiteness condition is satisfied automatically. In our case, the finiteness condition imposed  is  to guarantee the existence of minimal graded projective resolutions (see Proposition \ref{minimal projective resolution}).
\end{remark}

In the rest of this section we always assume that $A$ is a left finite $\mathbb{N}$-graded ring generated in degree $1$  with $A_0$ noetherian semiperfect,  $J$ is the graded Jacobson radical of $A$ and $S=A/J$.

\begin{example}
If $A$ has global dimension one, then $A$ is Koszul, as $0\to J\to A\to S\to 0$ is a minimal graded projective resolution of ${}_AS$.

If $A=kQ$, where $Q$ is a quiver of the following form, with at least one arrow of positive weight,
\begin{center}
\begin{tikzcd}
\bullet \arrow[r] & \bullet \arrow[r] & \cdots \arrow[r] & \bullet \arrow[lll, bend right]
\end{tikzcd}
\end{center}
 it follows from \cite[Proposition 6.6]{RR2} that $A$ is a generalized Artin-Schelter regular algebra (see Definition \ref{definition of GAS regluar}) of dimension one. In particular, if there is at least one arrow of weight zero and at least one arrow of weight one, then $A_0$ is a non-semisimple finite-dimensional algebra and $A$ is a locally finite graded algebra generated by $A_1$ over $A_0$. This is a simple example of Koszul rings.
\end{example}

Next, for later use, we show that any finitely generated quasi-Koszul module has a finitely generated minimal projective resolution.
\begin{lemma}\label{JM is fg}
If $M$ is a finitely generated graded $A$-module, then so is $JM$.
\end{lemma}
\begin{proof} Since $A_0$ is noetherian, $J(A_0)=A_0x_1+\cdots + A_0x_r$ for some $x_i\in A_0$.
Since  $A$ is left finite, we may assume that  $A_1=A_0y_1+\cdots+A_0y_s$ for some $y_j \in A_1$. Then $A_{\geqslant 1}=Ay_1+\cdots+Ay_s$ as $A$ is generated in degree $1$.   Therefore $J=Ax_1+\cdots + Ax_r+Ay_1 +\cdots +Ay_s$ is a finitely generated $A$-module.

Suppose that $M=Am_1+\cdots +Am_t$. Then $JM=Jm_1+\cdots + Jm_t=\Sigma Ax_im_j+\Sigma Ay_{i}m_j$ is a finitely generated $A$-module.
\end{proof}

\begin{proposition}\label{minimal is fg}
If ${}_AM$ is a finitely generated  quasi-Koszul module, then any module $P_n$ in the minimal graded projective resolution $P_\bullet \to M \to 0$ of $M$ is finitely generated.
\end{proposition}
\begin{proof}
Let $\cdots\to P_n\xrightarrow{d_n} P_{n-1}\to \cdots P_0 \xrightarrow{d_0} M\to 0$ be a minimal graded projective resolution of $M$. Since $M$ is finitely generated, $P_0$ is finitely generated.

By Lemma \ref{JM is fg}, $JP_0$ is finitely generated.
Since $J\Ker d_0=J^2P_0\cap \Ker d_0$, the natural map $\Ker d_0/J\Ker d_0\to JP_0/J^2P_0$ is injective. Hence $\Ker d_0/J\Ker d_0$ is a finitely generated $A/J$-module, and so it is a finitely generated $A$-module. By Nakayama's lemma, $\Ker d_0$ is finitely generated. The proof is completed by induction.
\end{proof}
\subsection{Generating property of quasi-Koszul modules and rings}
The following result characterizes quasi-Koszul modules (quasi-Koszul rings) in terms of Yoneda Ext-groups. The proof is based on the ideas in \cite[Sections 3 and 4]{GM1}.

Let $\Omega^i$ be the $i$-th syzygy functor.
\begin{theorem}\label{qK iff E(M) generated in degree 0}
Let $A$ be a left finite $\mathbb{N}$-graded ring generated in degree $1$  with $A_0$ noetherian semiperfect.
\begin{itemize}
\item[(1)] Suppose that $M$ is a left finite bounded below graded $A$-module. Then $M$ is quasi-Koszul if and only if, for any $n>0$, $$\gExt_A^1(S,S)\cdot \gExt_A^{n-1}(M,S)=\gExt^n_A(M,S).$$
\item[(2)] $A$ is a quasi-Koszul ring if and only if $\gExt_A^{\bullet}(S,S)$ is generated by  $\gExt^1_A(S,S)$ over $\gHom_A(S,S)$, that is, generated in degree $1$.
\end{itemize}
\end{theorem}

\begin{proof}
It suffices to prove the first statement. Suppose $M$ is quasi-Koszul. Let $\cdots\to P_n\xrightarrow{d_n} P_{n-1}\to \cdots \to P_0 \xrightarrow{d_0} M\to 0$ be a minimal graded projective resolution of $M$.
By Proposition  \ref{minimal projective resolution}, each $P_n$ is left finite.
It follows from the minimalism of the projective resolution  that
$$\gExt_A^n(M,S)=\gHom_A(P_n,S)\cong \gHom_A(\Omega^n(M),S).$$

Let $\Omega^n=\Omega^n(M)=\im d_n$. Suppose $g\in \gHom_A(\Omega^n,S)$ is a homogeneous element, say, of degree $t$. Since $g(J\Omega^n)=0$, $g$ factors through $\Omega^n/J\Omega^n$, that is, $g=\bar{g}p$ where $p: \Omega^n\to \Omega^n/J\Omega^n$ is the projection and $\bar{g}:\Omega^n/J\Omega^n\to S$ is induced by $g$.
\begin{center}
\begin{tikzcd}
\Omega^n \arrow[dd, "p"] \arrow[rr, "i", hook] \arrow[rd, "g"] &   & JP_{n-1} \arrow[dd] \arrow[ld, "g'"', dashed] \arrow[lldd, "p'", dashed, bend left=20] \\
                                                        & S &                                                                             \\
\Omega^n/J\Omega^n \arrow[rr, "\bar{i}"', tail] \arrow[ru, "\bar{g}", dashed] &   & JP_{n-1}/J^2P_{n-1}
\end{tikzcd}
\end{center}


Since $J\Omega^n=\Omega^n \cap  J^2P_{n-1}$ by assumption, the induced morphism $\bar{i}$ is injective. Since both $\Omega^n/J\Omega^n$ and $JP_{n-1}/J^2P_{n-1}$ are semisimple $S$-modules, $\bar{i}$ splits and there is a morphism $p':JP_{n-1}\to \Omega^n/J\Omega^n$ such that $p=p'i$ .
Let $g'=\bar{g}p'$. Then $g=g'i$.

Let $\pi:P_{n-1}\to P_{n-1}/JP_{n-1}$ be the natural projection. Then $\pi d_n=0$ and $\overline{\pi}\in \gExt_A^{n-1}(M,P_{n-1}/JP_{n-1})$.


We may assume that $P_{n-1}/JP_{n-1} = \mathop{\oplus}\limits_{j}S_j$ where $S_j$ are some graded simple $A$-modules. Since $P_{n-1}$ is left finite, there are only finitely many $S_j$ of degree $t$.  Suppose $Q_j\to P'_j\to S_j\to 0$ is the starting part of a minimal graded projective resolution of $S_j$ such that $P_{n-1} = \mathop{\oplus}\limits_{j} P'_j$ and $\mathop{\oplus}\limits_{j}Q_j \stackrel{\pi_2}{\to} JP_{n-1} \to 0$ is a projective cover of $JP_{n-1}$. Let $Q =\mathop{\oplus}\limits_{j}Q_j$.

Then we have the following commutative diagram
\begin{center}
\begin{tikzcd}
P_n \arrow[r, "\pi_1", two heads] \arrow[dd, "g\pi_1"', bend right] \arrow[d, "\exists h", dashed] & \Omega^n \arrow[d, "i", hook] \arrow[r, hook]    & P_{n-1} \arrow[rd, "\pi"] \arrow[d, "="] \arrow[r] & \cdots \arrow[r]           & M \arrow[r] & 0 \\
Q \arrow[r, "\pi_2", two heads] \arrow[d, "g'\pi_2"]                                               & JP_{n-1} \arrow[ld, "g'"] \arrow[r, hook] & P_{n-1} \arrow[r, "\pi"]                           & \frac{P_{n-1}}{JP_{n-1}} \arrow[r] & 0           &   \\
S                                                                                                  &                                           &                                                    &                            &             &
\end{tikzcd}
\end{center}
where $h$ is a lifting  of $\pi$, and  $d_n: P_n \to P_{n-1}$ factors through $\Omega^n$ via $\pi_1$. 




Now
$g'\pi_2 \in \gHom_A(Q,S) = \gExt_A^1(\dfrac{P_{n-1}}{JP_{n-1}}, S)$, which is of degree $t$. Since $Q =\mathop{\oplus}\limits_{j}Q_j$ is left finite, there is a finite subset $\{Q_1, Q_2, \cdots, Q_m\}$ of $\{Q_j\}$ such that $$g'\pi_2 \in \gHom_A(\mathop{\oplus}\limits_{j=1}^mQ_j, S) \cong \mathop{\oplus}\limits_{j=1}^m\gHom_A(Q_j,S) = \mathop{\oplus}\limits_{j=1}^m\gExt_A^1(S_j,S).$$

Let $g_j = g'\pi_2l_j: Q_j \to S$ and $h_j = p_jh: P_n \to Q_j$ for $1 \leqslant j \leqslant m$, where  $l_j: Q_j \to Q$ and $p_j: Q \to Q_j$ are the canonical injections and  projections respectively. It follows from $g\pi_1 = g'\pi_2 h$ that $g\pi_1=\sum\limits_{j=1}^m g'\pi_2 l_j p_jh=\sum\limits_{j=1}^m g_jh_j$, and by the definition of Yoneda products,
$$\gExt^n_A(M,S) \ni g\pi_1=\sum\limits_{j=1}^m g_j h_j\in \gExt_A^1(S_j,S)\cdot \gExt_A^{n-1}(M,S_j).$$


 By a degree shifting, we may assume that $S_j$ is concentrated in degree $0$, and so $\gExt_A^{n-1}(M,S_j)$ and $\gExt_A^1(S_j,S)$ can be viewed as direct summands of the groups $\gExt_A^{n-1}(M,S)$ and $\gExt_A^1(S,S)$ respectively. Therefore $$g\pi_1\in \gExt_A^1(S,S)\cdot \gExt_A^{n-1}(M,S).$$
It follows that
$\gExt_A^1(S,S)\cdot \gExt_A^{n-1}(M,S)=\gExt^n_A(M,S).$

Conversely, suppose $\gExt_A^1(S,S)\cdot \gExt_A^{n-1}(M,S)=\gExt^n_A(M,S)$ for all $n>0$.
Let $\Omega=\Omega(M)=\Ker d_0$.

First, we claim that any $g \in \gHom_A(\Omega,S)$ factors through $JP_0$.

Since
$g \in \gHom_A(\Omega,S) \cong \gExt_A^1(M,S)=\gExt_A^1(S,S)\cdot \gHom_A(M,S),$
$g\pi_1=\sum y_j\cdot  x_j$ for some  $x_j\in \gHom_A(M,S)$ and $y_j\in\gExt^1_A(S,S)$.

Let $Q \to J \to 0$ be a graded projective cover of $J$. Then
$$y_j\in\gExt^1_A(S,S) =\gHom_A(Q,S) \cong \gHom_A(J,S).$$
Now, for any $j$, we have the following commutative diagram
\begin{center}
\begin{tikzcd}
P_1 \arrow[d, "x_j^1"] \arrow[r, "\pi_1", two heads] & \Omega \arrow[r, hook] \arrow[d, "x_j'"]  & P_0 \arrow[d, "x_j^0"] \arrow[r] & M \arrow[r] \arrow[d, "x_j"] & 0 \\
Q \arrow[d, "y_j"] \arrow[r, two heads]              & J \arrow[ld, "y'_j"] \arrow[r, hook] & A \arrow[r]                      & S \arrow[r]                  & 0 \\
S                                                    &                                      &                                  &                              &
\end{tikzcd}
\end{center}
where $x_j^0,x_j^1,x_j'$ are the lifting of $x_j$. Then $g=\sum y_j'x_j'$.

Consider the following commutative diagram
\begin{center}
\begin{tikzcd}
0 \arrow[r] & \Omega \arrow[d, "x_j'"] \arrow[r] & P_0 \arrow[r] \arrow[d, "x_j^0"] & M \arrow[r] \arrow[d, "x_j"] & 0 \\
0 \arrow[r] & J \arrow[d, "y_j'"] \arrow[hook, r]  & A \arrow[d, "f_j"] \arrow[r]       & S \arrow[r] \arrow[d, "="]   & 0 \\
0 \arrow[r] & S \arrow[r, "\alpha"]         & E \arrow[r, "\beta"]             & S \arrow[r]                  & 0
\end{tikzcd}
\end{center}
where $E$ is the push-out of $J\hookrightarrow A$ and $J\xrightarrow{y_j'} S$.
Since $f_jx_j^0(JP_0)\subseteq \Ker \beta=\alpha(S)$, $y_j'x_j'$ factors through $JP_0$. It follows that $g$ factors through $JP_0$.
Hence $g(x)=0$ for any $x\in \Omega \cap J^2P_0$. Therefore
$$\Omega \cap J^2P_0\subseteq \cap \{\Ker g \mid g\in \gHom_A(\Omega,S)\}=J\Omega \subseteq \Omega \cap J^2P_0.$$
Hence $J\Ker d_0 = \Ker d_0 \cap J^2P_0$.

Note that
\begin{align*}
\gExt_A^1(\Ker d_0,S)&\cong \gHom_A(\Ker d_1,S)\\
&\cong \gExt_A^2(M,S)\\
&=\gExt^1_A(S,S)\cdot\gExt_A^1(M,S)\\
&=\gExt^1_A(S,S)\cdot\gHom_A(\Ker d_0,S).
\end{align*}
The proof is completed by replacing $M$ with $\Ker d_0$ and by induction.
\end{proof}


\section{Yoneda Ext rings and Koszul Duality}
In this section  we study the Yoneda Ext ring $\gExt_A^\bullet(S,S)$ of $A$, where $A$ is a left finite $\mathbb{N}$-graded ring generated in degree $1$  with $A_0$ noetherian semiperfect and $S=A/J$. We show that $A$ is a Koszul ring if and only if the Yoneda Ext ring $\gExt_A^\bullet(S,S)$ is a left finite classical Koszul ring
in the classical sense \cite[Definitions 1.2.1 and 1.2.4]{BGS1}. A Koszul duality theory is also given in this section.

\subsection{Some preparatory results}
The following lemmas and propositions were proved essentially in \cite[Section 5]{GM1}.
They are still true in our more general framework.

\begin{lemma}\label{fact 1}
Let $0\to X\to Y\to Z\to 0$ be an exact sequence of left finite bounded below graded $A$-modules such that $JX=X \cap JY$. If $P'$ and $P''$ are graded projective covers of $X$ and $Z$ respectively, then $P=P'\oplus P''$ is a graded projective cover of $Y$.
\end{lemma}
\begin{proof}
Obviously, there is  an  exact commutative diagram
\begin{center}
\begin{tikzcd}
0 \arrow[r] & P' \arrow[r] \arrow[d]        & P \arrow[r] \arrow[d]         & P'' \arrow[r] \arrow[d]       & 0 \\
0 \arrow[r] & X \arrow[d] \arrow[r]         & Y \arrow[d] \arrow[r]         & Z \arrow[d] \arrow[r]         & 0 \\
            & 0                             & 0                             & 0
\end{tikzcd}
\end{center}
Since $JX=X \cap JY$, we get the following  exact commutative diagram by tensoring the above diagram with $A/J\otimes_A -$.
\begin{center}
\begin{tikzcd}
0 \arrow[r] & P'/JP' \arrow[r] \arrow[d]        & P/JP \arrow[r] \arrow[d]         & P''/JP'' \arrow[r] \arrow[d]       & 0 \\
0 \arrow[r] & X/JX  \arrow[r]         &  Y/JY  \arrow[r]         &  Z/JZ \arrow[r]         & 0
\end{tikzcd}
\end{center}
It follows from $P'/JP'\cong X/JX$ and $P''/JP''\cong Z/JZ$ that $P/JP \cong Y/JY$. By Proposition \ref{minimal projective resolution} (1), both $P'$ and $P''$ are bounded below and left finite, so is $P$. Hence $P$ is a graded projective cover of $Y$.
\end{proof}

By Lemma  \ref{fact 1}, there is an exact commutative diagram, which is used frequently later in this paper.
\begin{equation}\label{comm-diagram-2}
\begin{tikzcd}
            & 0 \arrow[d]                   & 0 \arrow[d]                   & 0 \arrow[d]                   &   \\
0 \arrow[r] & \Omega(X) \arrow[r] \arrow[d] & \Omega(Y) \arrow[r] \arrow[d] & \Omega(Z) \arrow[r] \arrow[d] & 0 \\
0 \arrow[r] & JP' \arrow[r] \arrow[d]       & JP \arrow[r] \arrow[d]        & JP'' \arrow[r] \arrow[d]      & 0 \\
0 \arrow[r] & JX \arrow[d] \arrow[r]        & JY \arrow[d] \arrow[r]        & JZ \arrow[d] \arrow[r]        & 0 \\
            & 0                             & 0                             & 0                             &
\end{tikzcd}
\end{equation}

\begin{lemma}\label{exact sequence of syzygy}
Let $0\to X\to Y\to Z\to 0$ be an exact sequence of left finite bounded below graded $A$-modules with $JX =  X \cap JY$.  Suppose that $X$ is quasi-Koszul (resp. Koszul). Then for any $n > 0$,
$$0\to \Omega^n(X)\to \Omega^n(Y)\to \Omega^n(Z)\to 0$$
is exact, and $J\Omega^n(X)=\Omega^n(X) \cap J\Omega^n(Y)$ (resp. $J^k\Omega^n(X)=\Omega^n(X) \cap  J^k\Omega^n(Y)$ for all $k > 0$).
\end{lemma}
\begin{proof}
Suppose $X$ is quasi-Koszul. 
By applying the functor $A/J\otimes_A -$ to  diagram  \eqref{comm-diagram-2}, it follows from the quasi-Koszulity of $X$ that the following is an exact commutative diagram.

\begin{tikzcd}[column sep=small]
            & 0 \arrow[d]                               &                                           &                                           &   \\
            & A/J\otimes_A\Omega(X) \arrow[r] \arrow[d] & A/J\otimes_A\Omega(Y) \arrow[r] \arrow[d] & A/J\otimes_A\Omega(Z) \arrow[r] \arrow[d] & 0 \\
0 \arrow[r] & A/J\otimes_AJP' \arrow[r] \arrow[d]       & A/J\otimes_AJP \arrow[r] \arrow[d]        & A/J\otimes_AJP'' \arrow[r] \arrow[d]      & 0 \\
            & A/J\otimes_AJX \arrow[d] \arrow[r]        & A/J\otimes_AJY \arrow[d] \arrow[r]        & A/J\otimes_AJZ \arrow[d] \arrow[r]        & 0 \\
            & 0                                         & 0                                         & 0                                         &
\end{tikzcd}
Hence $A/J\otimes_A \Omega(X)\to A/J\otimes_A \Omega(Y)$ is injective, that is,  $J\Omega(X)= \Omega(X) \cap J\Omega(Y)$.
Therefore, the exact sequence $0\to \Omega(X)\to \Omega(Y)\to \Omega(Z)\to 0$ satisfies the same hypothesis as $0\to X\to Y\to Z\to 0$ do.
Hence the proof can be completed by induction.

For the Koszul case, it is left to prove that $J^k\Omega^n(X)=\Omega^n(X) \cap  J^k\Omega^n(Y)$ for all $k > 0$ which can be shown similarly  by applying the functor $A/J^k\otimes_A -$ instead of $A/J\otimes_A -$.
\end{proof}

\begin{corollary}\label{strongly quasi Koszul of exact sequence}
Let $0\to X\to Y\to Z\to 0$ be an exact sequence of left finite bounded below graded $A$-modules with  $JX =  X \cap JY$. If $X$ is Koszul (resp.  quasi-Koszul), then $Z$ is Koszul (resp. quasi-Koszul) if and only if  $Y$ is Koszul and $J^kX =  X \cap J^kY$ for all $k$ (resp. quasi-Koszul and $J^2\Omega^n(X) =  \Omega^n(X) \cap J^2\Omega^n(Y)$ for all $n \geqslant 0$).
\end{corollary}

\begin{proof} As in the proof of Lemma \ref{exact sequence of syzygy}, for the Koszul case,  we get a new commutative diagram by tensoring diagram  \eqref{comm-diagram-2} with the functor $A/J^k\otimes_A -$.
Then, by Snake lemma,  $A/J^k\otimes_A \Omega(Z)\to A/J^k\otimes_A JP''$ is injective if and only if $A/J^k\otimes_A \Omega(Y)\to A/J^k\otimes_A JP$ is injective and $J^kX =  X \cap J^kY$ for all $k$. Therefore, $J^k\Omega(Z)=\Omega(Z) \cap J^{k+1}P''$ if and only if $J^k\Omega(Y)=\Omega(Y) \cap  J^{k+1}P$ and $J^kX =  X \cap J^kY$ for all $k$.

Hence, the proof is completed inductively by Lemma \ref{exact sequence of syzygy} and by replacing $X,Y$ and $Z$ with $\Omega(X),\Omega(Y)$ and $\Omega(Z)$ respectively.

The proof of the quasi-Koszul case is similar.
\end{proof}


\begin{proposition}\label{JM is qK}
Let $A$ be a quasi-Koszul (resp. Koszul) ring.  If $M$ is a left finite bounded below Koszul module, then $JM$  is quasi-Koszul (resp. Koszul).
\end{proposition}
\begin{proof}
Since ${}_A(A/J)$ is quasi-Koszul (resp. Koszul),
$M/JM$ is quasi-Koszul (resp. Koszul).  Let $P$ be a graded projective cover of $M$. Since $M$ is Koszul and $M/JM$ is quasi-Koszul (resp. Koszul), $\Omega(M)$ is Koszul and $JP=\Omega(M/JM)$ is quasi-Koszul  (resp. Koszul).  Note that $0\to \Omega(M)\to JP\to JM\to 0$ is exact and $J^k\Omega(M)=\Omega(M) \cap J^{k+1}P$ for all $k \geqslant 0$. By Lemma \ref{exact sequence of syzygy}, $J^2\Omega^{n+1}(M)=\Omega^{n+1}(M) \cap J^2 \Omega^n(JP)$.
It follows from  Corollary \ref{strongly quasi Koszul of exact sequence} that $JM$ is quasi-Koszul (resp. Koszul).
\end{proof}

\begin{proposition}\label{exact sequence of ext group}
If $M$ is a left finite bounded below quasi-Koszul $A$-module, then for any $n>0$
\begin{itemize}
\item[(1)] $0\to \Omega^n(M)\to \Omega^n(M/JM)\to \Omega^{n-1}(JM)\to 0$ is exact and $J \,\Omega^n(M)=\Omega^n(M) \cap J\, \Omega^n(M/JM)$.
\item[(2)] $0\to \gExt^{n-1}_A(JM,S)\to \gExt_A^n(M/JM,S)\to \gExt^n_A(M,S)\to 0$ is exact.
\end{itemize}

If moreover $A$ is a Koszul ring,  $M$ is a left finite  and bounded below Koszul
$A$-module, then
\begin{itemize}
\item[$(1)'$] $0\to \Omega^n(J^{k-1}M)\to \Omega^n(J^{k-1}M/J^kM)\to \Omega^{n-1}(J^kM)\to 0$
is exact and $J^l\Omega^n(J^{k-1}M)=J^l\Omega^n(J^{k-1}M/J^kM)\cap \Omega^n(J^{n-1}M)$ for any $k, l >0$.
\item[$(2)'$]
\begin{small}
$0 \to \gExt^{n-1}_A(J^kM,S) \to \gExt_A^n(\dfrac{J^{k-1}M}{J^kM},S) \to \gExt^n_A(J^{k-1}M,S) \to 0$
\end{small}
is exact for any $k>0$.
\end{itemize}
\end{proposition}
\begin{proof}
(1) 
As noted before,  $0\to \Omega(M)\to \Omega(M/JM)\to JM\to 0$ is exact.
Since $M$ is quasi-Koszul and  $JP=\Omega(M/JM)$,  $$J\,\Omega(M)=\Omega(M) \cap  J^2P =\Omega(M) \cap J\,\Omega(M/JM).$$
Therefore, the conclusion follows from Lemma \ref{exact sequence of syzygy}.

(2) It follows from (1)  that 
$$0\to \dfrac{\Omega^n(M)}{J\Omega^n(M)}\to \dfrac{\Omega^n(M/JM)}{J\Omega^n(M/JM)}\to \dfrac{\Omega^{n-1}(JM)}{J\Omega^{n-1}(JM)}\to 0$$
is exact.
By acting on
 $0\to \Omega^n(M)\to \Omega^n(M/JM)\to \Omega^{n-1}(JM)\to 0$
with the functor
$$\gHom_A(-,S)\cong \gHom_A(-, \gHom_S(A/J, S))\cong \gHom_S(A/J \otimes_A -,S),$$
it implies that
\begin{small}
$$ 0\!\to\! \gHom_A(\Omega^{n-1}(JM),S)\!\to\! \gHom_A(\Omega^n(M/JM),S)\!\to\! \gHom_A(\Omega^n(M),S)\!\to\! 0$$
\end{small}
is exact, as $\gHom_S(-,S)$ is an exact functor.

Note that $\gExt_A^n(X, S)\cong \gHom_A(\Omega^n(X),S)$ for any graded $A$-module $X$ with a minimal graded projective resolution. Therefore,
$$0\to \gExt^{n-1}_A(JM,S)\to \gExt_A^n(M/JM,S)\to \gExt^n_A(M,S)\to 0$$
is exact for any $n>0$.

Note this means that the long exact Ext-group sequence induced by applying the functor $\gHom_A(-, S)$ to the short exact sequence $ 0 \to JM \to M \to M/JM \to 0$ is divided into short exact sequences.

If $A$ and $M$ are Koszul, then $J^kM$ is Koszul by Proposition \ref{JM is qK}. The proof of $(1)'$ and $(2)'$ are completed by replacing $M$ with $J^{k-1}M$.
\end{proof}

\subsection{Characterizations of Koszulity via Koszulity of Yoneda Ext rings}
 Let $\mathcal{E}=\gExt_A^\bullet(-,S)$ be the functor from the category of graded $A$-modules to the category of graded $E(A)$-modules. Note that $E(A)_0=\gExt_A^\bullet(S,S)_0=\gHom_A(S,S)\cong S^{op}$ is artinian semisimple.
The following is a generalized version of \cite[Theorems 6.1 and 9.1]{GM1}.

\begin{theorem}\label{A sqk implies E(A) Koszul}
Let $A$ be a left finite $\mathbb{N}$-graded ring generated in degree $1$  with $A_0$ noetherian semiperfect.
\begin{itemize}
\item[(1)] $A$ is a Koszul ring if and only if $E(A)$ is a 
 classical Koszul ring (in the sense of \cite[Defintion 1.2.1]{BGS1}).
\item[(2)] Suppose $A$ is Koszul. Then, for any left finite bounded below graded $A$-module $M$, $M$ is a Koszul $A$-module if and only if $\mathcal{E}(M)$ is a classical 
Koszul $E(A)$-module.
\end{itemize}
\end{theorem}
\begin{proof}
{\bf ``only if" part of (1)}. We show that $E(A)$ is a classical Koszul ring by constructing a linear projective resolution of $E(A)_0$ as an $E(A)$-module.


Let $\cdots\to P_n\xrightarrow{d_n} P_{n-1}\to \cdots \to P_0 \xrightarrow{d_0} S\to 0$ be a minimal graded projective resolution of $S$ as an $A$-module. Let  $P'_0 \xrightarrow{d'_0} E(A)_0\to 0$ be the projective cover of $E(A)_0$  as a graded $E(A)$-module, where  $P'_0 = E(A)$ and $\ker d_0'=J'=E(A)_{\geqslant 1}$ is the graded Jacobson radical of $E(A)$.
 Then
 $$ \ker d_0' = \mathop{\oplus}\limits_{i\geqslant 1}\gExt^i_A(S,S)=\mathop{\oplus}\limits_{i\geqslant 1}\gExt_A^{i-1}(JP_0,S)=\gExt_A^\bullet(JP_0,S)\,(-1)
 $$ where $(-)$ denotes the degree shifting of the homological grading.

 Since $JP_0/J^2P_0$ is a finite direct sum of graded simple modules,
 $\gExt_A^\bullet(JP_0/J^2P_0,S)$ is graded $E(A)$-projective.
 By Proposition \ref{exact sequence of ext group},
 $$0 \to \gExt_A^\bullet(J^2P_0,S)(-2) \to \gExt_A^\bullet(JP_0/J^2P_0,S)\,(-1) \to
 \ker d_0' \to 0$$
 is an exact sequence of graded $E(A)$-modules.

 Let $P'_1:= \gExt_A^\bullet(JP_0/J^2P_0,S)\,(-1)$.  It follows from Theorem \ref{qK iff E(M) generated in degree 0} that $P'_1$ is generated in degree $1$ as an $E(A)$-module, and
 $\gExt_A^\bullet(J^2P_0,S)(-2)  \subseteq  J'P_1'$.
 Hence  $P'_1 \to J'P'_0 \to 0$ is a graded projective cover of $J'P'_0$.

By using Proposition \ref{exact sequence of ext group} $(2)'$ repeatedly, one deduces that for any $k \geqslant 1$,
 \begin{equation}\label{reso-of-E(A)_0}
P'_k=
\gExt_A^\bullet(J^kP_0/J^{k+1}P_0,S)\,(-k)
 \end{equation}
 is generated in degree $k$. Thus $E(A)_0$ has a linear projective resolution
  \begin{equation}\label{reso-of-E(A)_0-long}
  \cdots\to P'_k \to P'_{k-1}\to \cdots \to P'_0 \to E(A)_0\to 0
 \end{equation}
 as a graded $E(A)$-module, and so $E(A)$ is a classical Koszul ring.

{\bf ``only if" part of (2)}. We first claim that $\gExt_A^\bullet(M/JM,S)$ is graded $E(A)$-projective if $M$ is bounded below and left finite.
As $M/JM=\oplus_{j\geqslant l} \bar{M_j}$ where $l=l(M)$ and $\bar{M_j}$ is a finite direct sum of graded simple modules concentrated in degree $j$, it suffices to prove
$\gExt_A^\bullet(\oplus_{j\geqslant l} S(-j), S)$ is graded $E(A)$-projective. Let $P_\bullet \to S \to 0$ be a minimal graded
projective resolution of ${}_AS$. Then $\oplus_{j\geqslant l}  P_\bullet(-j) \to \oplus_{j\geqslant l}  S(-j) \to 0$ is  a minimal graded projective resolution.

Note  that  $\Hom_{Gr}(\oplus_{j\geqslant l}  P_i(-j), S)=\oplus_{j\geqslant l}  \Hom_{Gr}( P_i(-j), S)$ for each $i$ as $P_i$ is bounded below.
Then
\begin{align*}\gExt_A^\bullet(\oplus_{j\geqslant l} S(-j), S) = &\oplus_i \gHom_A(\oplus_{j\geqslant l} P_i(-j), S)\\
=&\oplus_i (\oplus_n \Hom_{Gr}(\oplus_{j\geqslant l} P_i(-j)(-n), S))\\
=&\oplus_i (\oplus_n \Hom_{Gr}(\oplus_{j\geqslant l} P_i(-j), S)(n))\\
=&\oplus_i (\oplus_n (\oplus_{j\geqslant l} \Hom_{Gr}( P_i(-j), S))(n))\\
=&\oplus_i (\oplus_{j\geqslant l} (\oplus_n \Hom_{Gr}( P_i(-j), S))(n))\\
=&\oplus_i (\oplus_{j\geqslant l} \gHom_A( P_i(-j), S))\\
=&\oplus_{j\geqslant l} (\oplus_i \gHom_A( P_{i}(-j), S))\\
=&\oplus_{j\geqslant l} (\oplus_i \gExt^i_A( S(-j), S))\\
=&\oplus_{j\geqslant l} \gExt^\bullet_A( S(-j), S).
\end{align*}
This finishes the proof of the claim.

So, $\gExt_A^\bullet(J^{k-1}M/J^kM,S)$ is graded $E(A)$-projective for all $k \geqslant 1$.
 Now for any left finite Koszul module $M$, it follows from Proposition \ref{JM is qK} that $J^{k-1}M$ is left finite Koszul.
By Proposition \ref{exact sequence of ext group} $(2)'$,
\small{$$0\to \gExt^\bullet_A(J^kM,S)(-1)\to \gExt_A^\bullet(J^{k-1}M/J^kM,S)\to \gExt^\bullet_A(J^{k-1}M,S)\to 0$$}
is exact. So we can construct a linear $E(A)$-projective resolution for $\mathcal{E}(M)$. Therefore, $\mathcal{E}(M)$ is a classical Koszul $E(A)$-module.

We postpone the ``if part" proof of (1) and (2) of the theorem, for which a  modified version of Lemma \ref{exact sequence of syzygy} is needed.
\end{proof}

 A graded $A$-module $M$ with a  minimal graded projective resolution is called {\it $l$-quasi-Koszul}, where $l>0$ is an integer,  if its minimal projective resolution $P_\bullet$ satisfies that $J^k\Omega^n(M)=\Omega^n(M)\cap J^{k+1}P_{n-1}$ for all $0\leqslant k\leqslant l$ and $n>0$. Quasi-Koszul modules are exactly $1$-quasi-Koszul modules.

\begin{lemma}\label{fact 3}
Let $0\to X\to Y\to Z\to 0$ be an exact sequence of left finite bounded below graded $A$-modules with $JX=X\cap JY$. If $X$, $Y$ and $Z$ are all $l$-quasi-Koszul modules, then, for any $n \geqslant 0$,
$0\to \Omega^n(X)\to \Omega^n(Y)\to \Omega^n(Z)\to 0$ is exact, and
for any $1\leqslant  k\leqslant l+1$, $$J^k\Omega^n(X)=\Omega^n(X)\cap J^k\Omega^n(Y).$$
\end{lemma}

\begin{proof} Since $X$ is 1-quasi-Koszul, by Lemma \ref{exact sequence of syzygy}, for any $n > 0$,
$$0\to \Omega^n(X)\to \Omega^n(Y)\to \Omega^n(Z)\to 0$$ is exact, and  $J\Omega^n(X)=\Omega^n(X)\cap J\Omega^n(Y)$.

By applying the functor $A/J^i\otimes_A -$ to diagram  \eqref{comm-diagram-2}, it follows from the $l$-quasi-Koszulity of $X$, $Y$ and $Z$ that  $J^i\Omega(X)=\Omega(X)\cap J^i\Omega(Y)$, and
$J^{i+1}X = JX \cap J^{i+1} Y = X \cap JY \cap J^{i+1} Y =  X \cap J^{i+1} Y$ for $0\leqslant i\leqslant l$.

The proof is finished by replacing  $0\to X\to Y\to Z\to 0$ with the exact sequence
$0\to \Omega^n(X)\to \Omega^n(Y)\to \Omega^n(Z)\to 0.$
\end{proof}

Now we continue the proof of ``if part" of Theorem \ref{A sqk implies E(A) Koszul}.
\begin{proof}{\bf   ``if part" of (1)}.
Suppose $E(A)$ is a 
classical Koszul ring. Note $E(A)_0$ is semisimple. By \cite[Proposition 1.2.3]{BGS1}, $E(A)$ is generated in degree $1$. Thus by Theorem \ref{qK iff E(M) generated in degree 0}, $A$ is a quasi-Koszul ring.

To prove that $A$ is Koszul, it suffices to prove that all the graded simple $A$-modules are Koszul. This is done by proving the following claim  inductively.

Claim: all the graded simple $A$-modules $N$ and all $J^2Q$, where $Q$ are left finite bounded below graded projective $A$-modules, are $l$-quasi-Koszul for any integer $l$.

Since $A$ is quasi-Koszul, $N$ is quasi-Koszul, that is, $1$-quasi-Koszul.

For any left finite, bounded below graded projective $A$-module $Q$, $\mathcal{E}(Q)$ is a classical Koszul $E(A)$-module as $E(A)$ is a classical Koszul ring. It follows from the quasi-Koszulity of $Q$  and Proposition \ref{exact sequence of ext group} that
$$0\to \mathcal{E}(JQ)(-1)\to \mathcal{E}(Q/JQ)\to \mathcal{E}(Q)\to 0$$
is exact.  Since $\mathcal{E}(Q)$ is a classical Koszul module, $\mathcal{E}(JQ)(-1)=\Omega(\mathcal{E}(Q))$ is generated in degree $1$. By 
Theorem \ref{qK iff E(M) generated in degree 0} (1),
$JQ$ is a quasi-Koszul $A$-module. Suppose $J^{k-1}Q$ is quasi-Koszul. Then
$$0\to \mathcal{E}(J^{k}Q)(-1)\to \mathcal{E}(J^{k-1}Q/J^k Q)\to \mathcal{E}(J^{k-1}Q)\to 0$$
is exact, and $\mathcal{E}(J^k Q)(-k)=\Omega(\mathcal{E}(J^{k-1}Q)(-k+1))=\cdots=\Omega^k(\mathcal{E}(Q))$ which is generated in degree $k$. By Theorem \ref{qK iff E(M) generated in degree 0} (1) again, $J^k Q$ is quasi-Koszul for any $k\geqslant 0$ by induction.

So, our claim is true for $l=1$.

Now, assume all graded simple $A$-modules $N$ and all $J^2Q$ are $l$-quasi-Koszul for $l \geqslant 1$, where $Q$ are left finite, bounded below graded projective $A$-module $A$-modules.

Let $P_\bullet$ be a minimal graded projective resolution of $N$. Then $P_0$ is an indecomposable graded projective $A$-module.
Note $P_1 \to JP_0 \to 0$ is a projective cover, $\Omega(JP_0)=\Omega^2(N)$ and  $JP_1 = \Omega(JP_0/J^2P_0)$. Then
\begin{equation} \label{important-exact=seq}
0\to \Omega(JP_0) \to \Omega(JP_0/J^2P_0)\to J^2P_0\to 0
\end{equation}
 is an exact sequence, and all the modules in above exact sequence are $l$-quasi-Koszul by  hypothesis.

 Since $N$ is quasi-Koszul, $$J\Omega(JP_0)= J \Omega^2(N)= \Omega^2(N) \cap J^2P_1 = \Omega(JP_0) \cap J  \Omega(JP_0/J^2P_0).$$
Then Lemma \ref{fact 3} applies to the exact sequence \eqref{important-exact=seq}, and we have, for all $n \geqslant 1$ and $k\leqslant l+1$,
\begin{equation}\label{equation 1}J^k\Omega^n(JP_0)=\Omega^n(JP_0)\cap   J^k\Omega^n(JP_0/J^2P_0).
\end{equation}

Next, we show that $N$ is $l+1$-quasi-Koszul, that is, for any $n$ and $ k\leqslant l+1$,
 \begin{equation}\label{key equation}
 J^k\Omega^n(N)=\Omega^n(N) \cap J^{k+1}P_{n-1}.
\end{equation}

Since $\Omega(N)=JP_0$,
$J^k\Omega(N)=J^{k+1}P_0= \Omega(N) \cap J^{k+1}P_0$ for all $k \geqslant 0$.
So, \eqref{key equation} holds for $n=1$.

Since $\Omega(JP_0/J^2P_0)=JP_1$, by \eqref{equation 1}, for all $0\leqslant k\leqslant l+1$,
\begin{small}
\begin{equation}\label{equation 2}
J^k\Omega^2(N)=J^k\Omega(JP_0)=\Omega(JP_0) \cap J^k\Omega(JP_0/J^2P_0) =\Omega^2(N) \cap J^{k+1}P_1. 
\end{equation}
\end{small}
So, \eqref{key equation} holds for $n=2$.

By Lemma \ref{fact 1}, we have the following exact commutative diagram, where $P_2' \to J^2P_0 \to 0$ is a projective cover of $J^2P_0$ and ${P_2}''=P_2\oplus {P'}_2$.
\begin{center}
\begin{tikzcd}
            & 0 \arrow[d]                                    & 0 \arrow[d]                               & 0 \arrow[d]                        &   \\
0 \arrow[r] & \Omega^2(JP_0) \arrow[r] \arrow[d] & \Omega^2(JP_0/J^2P_0) \arrow[r] \arrow[d] & \Omega(J^2P_0) \arrow[r] \arrow[d] & 0 \\
0 \arrow[r] & P_2 \arrow[r] \arrow[d]                        & {P_2}'' \arrow[r] \arrow[d]                     & {P'}_2 \arrow[r] \arrow[d]             & 0 \\
0 \arrow[r] & \Omega(JP_0) \arrow[r] \arrow[d]   & \Omega(JP_0/J^2P_0) \arrow[r] \arrow[d]   & J^2P_0 \arrow[r] \arrow[d]         & 0 \\
            & 0                                              & 0                                         & 0                                  &
\end{tikzcd}
\end{center}

Note that \eqref{equation 2} holds for all graded simple module $N$. So, it holds for $N=JP_0/J^2P_0$. Then, by \eqref{equation 1} and \eqref{equation 2}, for all $0\leqslant k\leqslant l+1$,
\begin{align*}
J^k\Omega^3(N)&=J^k\Omega^2(JP_0)=\Omega^2(JP_0) \cap J^k\Omega^2(JP_0/J^2P_0) \\
&=\Omega^2(JP_0) \cap \Omega^2(JP_0/J^2P_0)\cap J^{k+1}{P_2}''\\
&=\Omega^2(JP_0) \cap J^{k+1}{P_2}''\\
&=\Omega^2(JP_0)\cap  P_2\cap J^{k+1}{P_2}''\\
&=\Omega^2(JP_0) \cap J^{k+1}P_2\\
&=\Omega^3(N) \cap J^{k+1}P_2.
\end{align*}
Hence, \eqref{key equation} holds for $n=3$.


Inductively,  we have, for any $n$ and $0\leqslant k\leqslant l+1$,
$$J^k\Omega^n(N)=\Omega^n(N) \cap J^{k+1}P_{n-1}.$$
Thus $N$ is an $(l+1)$-quasi-Koszul module.

Hence $\Omega(JP_0)= \Omega^2(N)$ and $\Omega(JP_0/J^2P_0)$ are $(l+1)$-quasi-Koszul. Similarly to the proof of Corollary \ref{strongly quasi Koszul of exact sequence}, we can show that $J^2P_0$ is $(l+1)$-quasi-Koszul.

Since $N$ is an arbitrary graded simple $A$-module and $P_0$ is its graded projective cover, the induction is completed.

Thus, every graded simple $A$-module is $l$-quasi-Koszul  for any $l>0$, so it is Koszul.

{\bf ``if part" of (2)}. By (1), $E(A)$ is a classical Koszul algebra, so it is generated in degree $1$. If $\mathcal{E}(M)$ is a 
classical Koszul module, it follows from Theorem \ref{qK iff E(M) generated in degree 0} (1) that $M$ is a quasi-Koszul module.

 Since $A$ is Koszul, it follows from  Proposition \ref{JM is qK} that $JQ$ is Koszul for all left finite bounded below graded projective $A$-module $Q$.

By Proposition \ref{exact sequence of ext group},
$0\to \Omega^n(M)\to \Omega^n(M/JM)\to \Omega^{n-1}(JM)\to 0$ is exact and $J \,\Omega^n(M)=\Omega^n(M) \cap J\, \Omega^n(M/JM)$ for any $n > 0$.

Now suppose that $M$ is $l$-quasi-Koszul for $l \geqslant 1$.
A similar proof to Proposition \ref{JM is qK} shows that $JM$ is $l$-quasi-Koszul. It follows from Lemma \ref{fact 3} that for any $n$ and $1 \leqslant k \leqslant l+1$,
 \begin{equation}
     J^k\Omega^n(M)=\Omega^n(M)\cap J^k\Omega^n(M/JM).
 \end{equation}
 Let $Q_\bullet \to M \to 0$ and  $Q''_\bullet \to JQ_0 \to 0$ be minimal graded projective resolutions of $M$ and $JQ_0$ respectively. Then, it follows from the exact sequence $0 \to \Omega(M) \to JQ_0 \to JM \to 0$ and the Koszulity of $JQ_0$ that, for any $1 \leqslant k \leqslant l+1$,
 \begin{align*}
 J^k\Omega^n(M)=&\Omega^n(M)\cap J^k\Omega^n(M/JM)\\
 =&\Omega^n(M)\cap J^k\Omega^{n-1}(JQ_0)\\
  =&\Omega^n(M)\cap \Omega^{n-1}(JQ_0) \cap J^{k+1}  Q''_{n-2}\\
 =&\Omega^n(M)\cap Q_{n-1} \cap J^{k+1} Q''_{n-2}\\
 =&\Omega^n(M)\cap J^{k+1} Q_{n-1}.
 \end{align*}
Hence, $M$ is $(l+1)$-quasi-Koszul. This finishes the proof.
\end{proof}

\begin{remark} \label{E(A)-left-finite}
It should be noted that if a quasi-Koszul module $M$ is finitely generated then $\mathcal{E}(M)$ is a left finite  $E(A)$-module.
Let $Q_\bullet \to M \to 0$ be a minimal graded projective resolution of $M$. It follows from Proposition \ref{minimal is fg} that every $Q_n$ is finitely generated.
Therefore
$$\gExt_A^n(M,S)=\gHom_A(Q_n, S) \cong \gHom_S(Q/JQ, S)$$ is a finitely generated $S^{op}$-module. It is easy to see that the $E(A)_0$-module structure of $\gExt_A^n(M,S)$ given by the Yoneda product corresponds to the canonical $S^{op}$-module structure of $\gHom_A(J,S)$. Hence $\gExt_A^n(M,S)$ is finitely generated as an $E(A)_0$-module. In particular, $E(A)$ is a left finite classical Koszul algebra if $A$ is a Koszul algebra.

But, in general, $\mathcal{E}(M)$ will not be left finite as an $E(A)$-module, even when $M$ is Koszul (for example, $M=\oplus_{i\geqslant 0} A(-i)$).
\end{remark}

\subsection{Koszul Duality for rings and modules}
In the classical Koszul theory, if $A$ is left finite classically Koszul then $E(A)$ is left finite classically Koszul and the Yoneda Ext ring of $E(A)$ is isomorphic to $A$ \cite[Theorem 1.2.5]{BGS1}. In our setting, the Koszul dual of the Koszul dual of $A$ will not be $A$ in general but $\Gr_J A$ (see Theorem \ref{FE(S) cong Grj A}). If $A_0$ is semisimple, it recovers the classical results.

As defined in the Introduction, let
\begin{align*}
&\mathcal{E}=\gExt_A^\bullet(-,A/J),\\
&\mathcal{F}=\gExt_{E(A)}^\bullet(-,E(A)/J_{E(A)}),\\ 
&\mathcal{G}=\gExt^\bullet_{\Grj A}(-,A/J).
\end{align*}
Note that when we consider the Koszulity, $E(A)$ and $\Grj A$ are viewed  as a graded ring via the homological degree and the graded degree induced by the $J$-adic filtration respectively, although both $E(A)$ and $\Grj A$ are bigraded rings.

Let $\mathcal{K}_A$,  $\mathcal{K}_{E(A)}$ and $\mathcal{K}_{\Grj A}$ be the full subcategories  of finitely generated Koszul $A$-modules, classical Koszul ${E(A)}$-modules and classical Koszul ${\Grj A}$-modules  respectively in the corresponding categories.

Now we are ready to prove Theorem \ref{theorem 3} which is a generalized version of Koszul algebra duality and Koszul module duality.

\begin{theorem}\label{FE(S) cong Grj A}
Let $A$ be a Koszul ring. Then
\begin{itemize}
\item[(1)] $E(E(A))\cong \Grj A$ as graded rings (in fact, as bigraded rings).
\item[(2)] The functors  $\mathcal{E}$ and  $\mathcal{F}$ restrict to
$\xymatrix{
\mathcal{K}_A\ar[r]^{\mathcal{E}} &\mathcal{K}_{E(A)}\ar[r]^{\mathcal{F}}& \mathcal{K}_{\Grj A},
}$
 such that, for any $M\in \mathcal{K}_{A}$, $\mathcal{F}\mathcal{E}(M)\cong \Grj M$ as graded $\Grj A$-modules.
\item[(3)] The functors $\mathcal{F}$ and  $\mathcal{G}$ restrict to
$\xymatrix{
\mathcal{K}_{E(A)}\ar@<1mm>[r]^{\mathcal{F}}& \mathcal{K}_{\Grj A}\ar@<1mm>[l]^{\mathcal{G}}
}$, which gives a duality of categories.
 \end{itemize}
\end{theorem}
\begin{proof} (1) It follows from Lemma \ref{JM is fg} that $J^i/J^{i+1}$ is finitely generated  as an $S$-module for all $i \geqslant 0$.
So, there is a canonical $S$-module isomorphism given by the evaluation map
$$J^i/J^{i+1} \to \Hom_{S^{op}}(\Hom_S(J^i/J^{i+1}, S), S), \, \bar{x} \mapsto \big(\bar{x}^{**}: \varphi \to \varphi(\bar{x})\big).$$
By Theorem \ref{A sqk implies E(A) Koszul}, $E(A)$ is a classical Koszul ring. Let $P_\bullet' \to E(A)_0 \to 0$ be the minimal graded projective resolution of $E(A)_0$ given in the proof of Theorem \ref{A sqk implies E(A) Koszul}. Then $P_i'$
is generated in degree $i$ for all $i$. By using the equality \eqref{reso-of-E(A)_0} in the proof of Theorem \ref{A sqk implies E(A) Koszul} and the fact $E(A)_0 \cong S^{op}$,
\begin{align*}
&\gExt^i_{E(A)}(E(A)_0,E(A)_0)\\
=&\gHom_{E(A)}(P_i',E(A)_0)\\
\cong & \Hom_{E(A)_0}((P_i')_i, E(A)_0)\\
= &\Hom_{S^{op}}(\gHom_S(J^i/J^{i+1},S), S)\\
\cong & J^i/J^{i+1}
\end{align*}
as $S$-modules.
%
It induces a graded $S$-module isomorphism $$\theta: E(E(A))\to \mathop{\oplus}\limits_{i\geqslant 0}J^i/J^{i+1},\,  g \mapsto \overline{y}$$
for any $g\in \gHom_{E(A)}(P_i',E(A)_0)$, where $\overline{y} \in J^i/J^{i+1}$ satisfies that  $\overline{y}^{**}$ is equal to the action of
$g$ restricting to the degree $0$ part.

In fact, $\theta$ is a graded ring isomorphism as we show next.
For any
$$f\in \gExt_{E(A)}^j(E(A)_0,E(A)_0) \cong \Hom_{S^{op}}(\gHom_S(J^j/J^{j+1},S),S)$$
and 
$$g\in \gExt^i_{E(A)}(E(A)_0,E(A)_0)\cong \Hom_{S^{op}}(\gHom_S(J^i/J^{i+1},S),S)$$
with $\theta(f)= \overline{x}\in J^j/J^{j+1}$ and $\theta(g)= \overline{y}\in J^i/J^{i+1}$,
we have to show $\theta(f\cdot g)= \theta(f)\, \theta(g)$, that is, to show that $(f\cdot g)(\varphi)=\varphi(\overline{xy})$ for any $\varphi\in \gHom_S(J^{j+i}/J^{j+i+1},S)$.

Consider first the case that $j=0$.
For any 
$\varphi\in
\gHom_S(J^i/J^{i+1},S), $
$g(\varphi)=\varphi(\overline{y})\in S,$
and for any $\psi \in
\gHom_S(A/J,S), $
$f(\psi)=\psi(\overline{x})\in S$. It follows from the definition of Yoneda products that
$$(f\cdot g)(\varphi)=f(\varphi(\overline{y}))=\overline{x}\varphi(\overline{y})=\varphi(\overline{xy}).$$
Therefore, $\theta(f\cdot g)= \overline{xy} =\theta(f)\, \theta(g)\in J^i/J^{i+1}$.





Consider next the case that $j=1$. For any 
$$\varphi\in \gHom_S(J^{i+1}/J^{i+2},S) \cong \gHom_A(J^{i+1},S),$$
$\varphi$ corresponds to an element of $P_i'$ via the injective map
\begin{equation}\label{description-of-phi}
    0 \to \gHom_A(J^{i+1},S)  \to \gExt_A^1(J^i/J^{i+1},S) = (P'_i)_{i+1}
\end{equation}
given by Proposition \ref{exact sequence of ext group} $(2)'$. Let us describe this exact sequence.

Let $\pi: Q\to J^i \to 0$ be a graded projective cover of $J^i$. Then $\Omega(J^i/J^{i+1})=JQ$, and there is an exact commutative diagram
\begin{center}
\begin{tikzcd}
            & 0 \arrow[r]       & \Omega(J^i) \arrow[d] \arrow[r] & \Omega(J^i/J^{i+1}) \arrow[d] \arrow[r, "\pi"] & J^{i+1} \arrow[r] & 0 \\
            &                   & Q \arrow[d, "\pi"] \arrow[r]           & Q \arrow[d]                                    &                   &   \\
0 \arrow[r] & J^{i+1} \arrow[r] & J^i \arrow[r]                   & J^i/J^{i+1} \arrow[r]                          & 0.                 &
\end{tikzcd}
\end{center}

It follows from the construction of the long exact sequence of Ext-groups that the exact sequence \eqref{description-of-phi} is
$$0 \to \gHom_A(J^{i+1},S)  \xrightarrow{{\pi}^*} \gExt_A^1(J^i/J^{i+1},S) = (P'_i)_{i+1}.$$

Hence $\varphi\in \gHom_S(J^{i+1}/J^{i+2},S)$ corresponds to the element $\alpha:={\pi}^*(\varphi)=\varphi\circ \pi \in (P'_i)_{i+1}= \gHom_A(\Omega(J^i/J^{i+1}),S)$ of $P_i'$.

Suppose $J^i/J^{i+1}= \mathop{\oplus}\limits_{k=1}^t S_k$ where $S_k$ are some graded simple $A$-modules, and $Q = \mathop{\oplus}\limits_{k=1}^t Q_k$ with $Q_k$ being a graded projective cover of $S_k$.
Let $p_k:J^i/J^{i+1} \to S_k$ and $p'_k:JQ \to JQ_k$ be the projections induced by the canonical projection $Q \to Q_k$.
Let $\alpha_k \in\gHom_A(JQ_k,S)$ be the composition $JQ_k \subseteq JQ \stackrel{\alpha}{\rightarrow}S$.

By the definition of Yoneda products and the following commutative diagram
\begin{center}
\begin{tikzcd}[column sep=small]
0 \arrow[r] & \Omega(J^i/J^{i+1}) \arrow[r] \arrow[d, "="] \arrow[d] & Q \arrow[r] \arrow[d, "="]  & J^i/J^{i+1} \arrow[r] \arrow[d, "="]          & 0 \\
0 \arrow[r] & JQ=\Omega(Q/JQ) \arrow[r] \arrow[d, "p'_k"]                                            & Q \arrow[r] \arrow[d]  & Q/JQ\cong J^i/J^{i+1} \arrow[r] \arrow[d, "p_k"] & 0 \\
0 \arrow[r] & J(Q_k)= \Omega(S_k) \arrow[r] \arrow[d, "\alpha_k"]                                       & Q= Q_k \arrow[r] &  S_k \arrow[r]                      & 0 \\
            & S                                                                            &                        &                                           &
\end{tikzcd}
\end{center}
we have $\alpha =\sum_k \alpha_k p'_k= \sum_k \alpha_k \cdot p_k \in \gExt^1_A(S_k,S) \cdot \gHom_A(J^i/J^{i+1},S_k)$.


By taking a suitable grading shift, we may view
$$\alpha = \sum_k \alpha_k \cdot p_k \in \gExt^1_A(S,S) \cdot \gHom_A(J^i/J^{i+1},S).$$

Suppose $\tilde{y}\in Q$ such that $\pi(\tilde{y}) = y \in J^i$.
Then $\pi(x\tilde{y})=xy \in J^{i+1}$. Let
 $\tilde{y}=\sum\limits_{k=1}^t \tilde{y}_k$ such that $\tilde{y}_k\in Q_k$ and $\overline{\pi(\tilde{y}_k)}= p_k(\overline{y})\in S_k$.
Hence
\begin{equation}\label{action-of-phi}
\varphi(\overline{xy})=\varphi(\overline{\pi(x\tilde{y})})=\alpha(x\tilde{y})=\sum\alpha_k(x\tilde{y}_k).
\end{equation}

To finish the proof of the $j=1$ case, it is left to show
\begin{equation} \label{action-of-phi-2}
   (f\cdot g)(\varphi)=\sum\alpha_k(x\tilde{y}_k).
\end{equation}

Now, consider the following commutative diagram
\begin{center}

\begin{tikzcd}
P_{i+1}' \arrow[r, two heads] \arrow[d, "g_1"]          & \Omega^{i+1}(E(A)_0) \arrow[d, "g_0'"] \arrow[r, hook] & P'_i \arrow[d, "g_0"]  \arrow[r, "d'_i"] & \Omega^i(E(A)_0)  \arrow[d, "g"] \arrow[r] & 0 \\
P'_1 \arrow[r, phantom] \arrow[r, two heads] & \Omega^1(E(A)_0) \arrow[r, hook] \arrow[d, "f"]        & P'_0 \arrow[r]                  & E(A)_0 \arrow[r]                          & 0 \\
                                                        & E(A)_0                                                 &                                 &                                           &
\end{tikzcd}
\end{center}
where $g_0,g_0'$ and $g_1$ are the lifting of $g$. Since $P'_i$ is generated in degree $i$, we may assume that when restricting to degree $i$ part the action of $g_0$ is the same as that of $g$. Note that $g_0'$ is the restriction of $g_0$ on $\Omega^{i+1}(E(A)_0)$, $\Omega^{i+1}(E(A)_0)= \gExt^\bullet_A(J^{i+1}, S)[-i-1]$ and  $$\Omega^{i+1}(E(A)_0)_{i+1}= \big(\gExt^\bullet_A(J^{i+1}, S)[-i-1]\big)_{i+1}=\gHom_A(J^{i+1}, S).$$
By the definition of Yoneda products,
$$(f\cdot g)(\varphi)=(f\circ g_0')(\varphi)=f(g_0(\alpha))=f(g_0(\sum \alpha_k\cdot p_k)).$$

Since $g_0$ is $E(A)$-linear,
$$f(g_0(\sum \alpha_k\cdot p_k))=f(\sum\alpha_k\cdot g_0(p_k))=f(\sum\alpha_k\cdot  p_k(\overline{y})_r)$$
where $p_k(\overline{y})_r \in \gHom_S(A/J, S_k)$ is the right multiplication of $p_k(\overline{y}) \in S_k$.

 To see the Yoneda product of $\alpha_k$ with $p_k(\overline{y})_r$, we need the following commutative diagram
\begin{center}
\begin{tikzcd}
0 \arrow[r] & JQ_k \arrow[d, "(\tilde{y}_k)_r"] \arrow[r] & Q_k \arrow[d, "(\tilde{y}_k)_r"] \arrow[r] & S_k \arrow[d, "(p_k(\overline{y}))_r"] \arrow[r] & 0 \\
0 \arrow[r] & JQ_k \arrow[d, "\alpha_k"] \arrow[r]        & Q_k \arrow[r]                              & S_k \arrow[r]                                      & 0 \\
            & S                                       &                                          &                                                  &
\end{tikzcd}
\end{center}
where $(*)_r$ is the right multiplication given by $*$. Then $$\alpha_k\cdot p_k(\overline{y})=\alpha_k\circ (\tilde{y}_k)_r$$
and $f(\sum\alpha_k\cdot p_k(\overline{y}))=f(\sum\alpha_k\circ (\tilde{y}_k)_r)
   =(\sum\alpha_k\circ (\tilde{y}_k)_r)(x)=\sum\alpha_k(x\tilde{y}_k)$.
Hence
$ (f\cdot g)(\varphi)=\sum\alpha_k(x\tilde{y}_k)$, that is, \eqref{action-of-phi-2} holds.

Therefore $(f\cdot g)(\varphi)=\varphi(\overline{xy})$ and so $\theta(f\cdot g)=\overline{xy}=\theta(f) \theta(g)$. The proof of the $j=1$ case is finished.

Finally, consider the general case that $j > 1$.  By Theorem \ref{qK iff E(M) generated in degree 0},
$$\gExt^{j}_{E(A)}(E(A)_0,E(A)_0) \ni f=\sum\limits_{p=1}^s h_p\cdot f_p$$
for some $h_p\in\gExt^1_{E(A)}(E(A)_0, E(A)_0)$ and $f_p\in \gExt^{j-1}_{E(A)}(E(A)_0,E(A)_0)$.

Suppose
$\theta(h_p)=\overline{x}_p\in J/J^2$  and $\theta(f_p)=\overline{x}'_p\in J^{j-1}/J^{j}.$
Then, by the $j=1$ case already proved and by induction hypothesis,
$$\theta(f)=\sum \overline{x_px'_p}\in J^{j}/J^{j+1} \textrm{ and } \theta(f_p\cdot g)=\overline{x'_py}\in J^{i-1+j}/J^{i+j}.$$
Thus by  the $j=1$ case again,
 $$\theta(f\cdot g)=\theta(\sum h_p\cdot (f_p\cdot g))=\sum\overline{x_p}\cdot \overline{x'_py}=\sum\overline{x_px'_py}=\theta(f)\theta(g).$$
It follows that $\theta$ is an isomorphism of graded rings.





(2) By definition, $\mathcal{F}$  is a functor from the category of graded $E(A)$-modules to the category of graded $E(E(A))$-modules.  Since $E(E(A)) \cong \Gr_J A$ as graded rings by (1), $\mathcal{F}$ can be naturally viewed as a functor from the category of graded $E(A)$-modules to the category of graded $\Gr_J A$-modules. It follows from the linear $E(A)$-projective resolution
$$ \cdots \to \mathcal{E}(J^iM/J^{i+1}M)[-i]\to \cdots \to  \mathcal{E}(M/JM) \to \mathcal{E}(M) \to 0$$
(see the proof  of Theorem \ref{A sqk implies E(A) Koszul} (2)) that
\begin{align*}
\mathcal{F}(\mathcal{E}(M))_i & = \gHom_{E(A)}(\mathcal{E}(J^iM/J^{i+1}M)[-i], E(A)_0)\\
& \cong \Hom_{E(A)_0}(\gHom_S(J^iM/J^{i+1}M, S), E(A)_0)\\
 & = \Hom_{S^{op}}(\gHom_S(J^iM/J^{i+1}M, S),S) \\
 & \cong J^iM/J^{i+1}M,
\end{align*}
where the $\cong$ holds because $J^iM/J^{i+1}M$ is finitely generated. 
A similar argument to (1) shows that $\mathcal{F}(\mathcal{E}(M))\cong \Grj M$ as graded $\Grj A$-modules.

(3) Since both $E(A)_0$ and $(\Grj A)_0$ are artinian semisimple, by replacing $A$ with $E(A)$ and $\Grj A$ respectively and repeating the proof above, it follows that $\mathcal{F},\mathcal{G}$ give a duality between $\mathcal{K}_{E(A)}$ and $\mathcal{K}_{\Grj A}$.
Or by Theorem \ref{A sqk implies E(A) Koszul} and Remark \ref{E(A)-left-finite}, $E(A)$ is left finite classical Koszul with  $E(A)_0$ artinian semisimple, then the conclusion follows from the classical Koszul theory.
\end{proof}

Under the assumptions in Theorem \ref{FE(S) cong Grj A}, if furthermore $A_0$ is semisimple, then $A\cong \Grj A$ as graded rings. Therefore, Theorem  \ref{FE(S) cong Grj A} reduces to the classical Koszul duality (\cite[Theorem 5.2]{GM2}, \cite[Theorem 2.10.2]{BGS1}).

The following result gives the converse statements of Theorem \ref{FE(S) cong Grj A} (1) and (2) in some sense.
\begin{theorem}\label{FE(M) cong Grj M imply M sqk}
Let $A$ be a left finite $\mathbb{N}$-graded ring generated in degree $1$  with $A_0$ noetherian semiperfect.
\begin{itemize}
\item[(1)] If $E(A)$ is generated in degree $1$  and $E(E(A))\cong \Grj A$ as graded rings, then $A$ is Koszul.
\item[(2)] Suppose that $A$ is Koszul and $M$ is a finitely generated graded $A$-module. If $\mathcal{E}(M)$ is generated in degree $0$ and $\mathcal{F}(\mathcal{E}(M))\cong \Grj M$ as graded $\Grj A$-modules, then $M$ is Koszul.
\end{itemize}
\end{theorem}
\begin{proof} (1) Since $E(A)$ is generated in degree $1$, it follows from Theorem \ref{qK iff E(M) generated in degree 0} that $A$ is quasi-Koszul. Then $E(A)$ is left finite as we see in Remark \ref{E(A)-left-finite}.
If $E(E(A))\cong \Grj A$, then it is generated in degree $1$. Thus by Theorem \ref{qK iff E(M) generated in degree 0} again, $E(A)$ is quasi-Koszul. By Proposition \ref{qk and classical koszul}, $E(A)$ is a classical Koszul ring. Therefore $A$ is Koszul by Theorem \ref{A sqk implies E(A) Koszul}.

(2) The proof is similar to that of (1).
\end{proof}

\section{More characterizations of Koszul property}
In this section
we first prove that if $M$ is a Koszul $A$-module, then $\Grj M$ is a classical Koszul $\Grj A$-module. The converse statement is true under an additional condition that $J(A_0)$ is nilpotent.  As a corollary, it is proved that $A$ is a Koszul ring if and only if so is $A^{op}$. More characterizations of the Koszulity are given  under the condition that $A_0$ is artinian.

\subsection{(Quasi-)Koszulity of A versus Koszulity of GrA }
The following result is trivial if $A_0$ is semisimple.
\begin{theorem}\label{A sqk implies Gr A Koszul}
Let $A$ be a left finite $\mathbb{N}$-graded ring generated in degree $1$  with $A_0$ noetherian semiperfect.
\begin{enumerate}
\item If $M$ is a Koszul $A$-module, then $\Grj M$ is a classical Koszul $\Grj A$-module.
\item If $A$ is a Koszul ring, then $\Grj A$ is a classical Koszul ring.
\end{enumerate}
\end{theorem}
\begin{proof} It suffices to prove (1).
Let $P_\bullet\to M\to 0$ be a minimal graded projective resolution of ${}_AM$. Since $M$ is Koszul, it follows that, with the  $J$-adic filtration,
$$0\to \Ker d_0[-1] \xrightarrow{i_0} P_0\xrightarrow{d_0} M\to 0$$
is strict exact, where $\Ker d_0[-1]$ is the shift of the $J$-adic filtration.
Therefore, there is an exact sequence of $\Grj A$-modules
$$0\to \Grj \Ker d_0[-1]\to \Grj P_0\to \Grj M\to 0.$$
It is easy to see that $\Grj P_0$ is a graded projective cover of the $\Grj A$-module $\Grj M$. By replacing $M$ with $\Ker d_0[-1]$, and doing this repeatedly, we can construct a minimal graded projective resolution of the $\Grj A$-module $\Grj M$:
$$\cdots\to \Grj P_i[-i]\to\cdots\to \Grj P_0\to \Grj M\to 0$$
which is a linear projective resolution. Thus $\Grj M$ is a classical Koszul $\Grj A$-module.
\end{proof}


A graded left ideal $I$ of an $\mathbb{N}$-graded ring $A$ is called degree-wise nilpotent, if for any positive integer $t$, there is an integer $N$ such that for all $n\geqslant N$, $(I^n)_{\leqslant t}=0$.

\begin{lemma}\label{locally nilpotent}
Let $A$ be an $\mathbb{N}$-graded ring and $I$ be a graded left ideal of $A$.
Then $I$ is degree-wise nilpotent if and only if $I_0$ is nilpotent.
\end{lemma}
\begin{proof}
For any integer $l \geqslant 0$, the degree $l$ part of $I^n$ has the form
$$(I^n)_l=\sum\limits_{i_1+\cdots+i_{s+1}=n-s, j_1 + \cdots + j_s =l}I_0^{i_1}I_{j_1}I_0^{i_2}I_{j_2}\cdots I_0^{j_s}I_{j_s}I_0^{i_{s+1}}$$
where $j_1, \cdots, j_s \geq 1$. Note that $i_1+\cdots+i_{s+1}=n-s \geqslant n - l$.

If $I_0$ is nilpotent, then, for any fixed positive integer $t$, $(I^n)_{\leqslant t}=0$ for sufficiently large $n$.

The other direction is trivial.
\end{proof}


Hence, the graded Jacobson radical $J$ of $A$ is degree-wise nilpotent if and only if $J(A_0)$ is nilpotent, under the hypothesis of Theorem \ref{A sqk implies Gr A Koszul}. Recall that a ring is noetherian semiperfect with nilpotent Jacobson radical if and only if it is a noetherian perfect ring if and only if it is an artinian ring.

\begin{theorem}\label{GrM is Koszul implies M is qK}
Suppose that $A$ is a left finite $\mathbb{N}$-graded ring generated in degree $1$ such that $A_0$ is artinian.
\begin{enumerate}
\item If $M$ is a left finite bounded below graded $A$-module such that $\Grj M$ is a classical Koszul $\Grj A$-module, then $M$ is a Koszul $A$-module.
\item If $\Grj A$ is a classical Koszul ring, then $A$ is a Koszul ring.
\end{enumerate}
\end{theorem}
\begin{proof} It suffices to prove (1).
Let $P_\bullet\to M\to 0$ be a minimal graded projective resolution of ${}_AM$.
Then $\Grj P_0 \to \Grj M\to 0$ is a graded projective cover of $\Grj M$.

Consider the exact sequence
$0\to \Ker d_0\xrightarrow{i_0} P_0\xrightarrow{d_0} M\to 0$
which is strict exact if we endow  $P_0$, $M$ with the $J$-adic filtration and $\Ker d_0$ with the induced submodule filtration, that is, $F_n\Ker d_0=\Ker d_0 \cap J^nP_0$. Then
$$0\to \Gr \Ker d_0\to \Grj P_0\to \Grj M\to 0$$
 is exact, where $\Gr \Ker d_0=\oplus (F_n\Ker d_0/F_{n+1}\Ker d_0)$.

Since $\Grj M$ is classically Koszul, $\Gr \Ker d_0$ is generated in degree $1$. Therefore
\begin{small}
$$
\dfrac{\Ker d_0 \cap J^nP_0}{\Ker d_0 \cap J^{n+1}P_0}=\dfrac{J^{n-1}}{J^n} \dfrac{\Ker d_0}{\Ker d_0 \cap J^2 P_0}=\dfrac{J^{n-1}\Ker d_0+\Ker d_0 \cap J^{n+1}P_0}{\Ker d_0 \cap J^{n+1}P_0}.
$$
\end{small}
Hence
\begin{align*}
\Ker d_0 \cap J^nP_0 &=J^{n-1}\Ker d_0+ \Ker d_0 \cap J^{n+1}P_0 \\
&=J^{n-1}\Ker d_0+J^{n}\Ker d_0+\Ker d_0 \cap J^{n+2}P_0 \\
&=\cdots\\
&=J^{n-1}\Ker d_0+\Ker d_0 \cap J^{n+m}P_0
\end{align*}
for any positive integers $n$ and $m$.

Since $A_0$ is artinian, it follows from Lemma \ref{locally nilpotent} that $J$ is degree-wise nilpotent. Hence, for any fixed $n$ and any fixed $t$, there is a  large enough integer $m$ such that $(J^{n+m}P_0)_{\leqslant t}=0$. Therefore
\begin{align*}
(\Ker d_0 \cap J^nP_0 )_{\leqslant t}&=(J^{n-1}\Ker d_0)_{\leqslant t}+ (\Ker d_0 \cap J^{n+m}P_0)_{\leqslant t}\\
&=(J^{n-1}\Ker d_0)_{\leqslant t}.
\end{align*}
By the arbitrariness of $t$,
$J^{n-1}\Ker d_0 = \Ker d_0 \cap J^nP_0$ for all $n \geqslant 1$.


Replacing $M$ by $\Ker d_0$ and by induction, it follows that
$$J^{n-1}\Ker d_i=\Ker d_i \cap J^nP_i$$ for all $n \geqslant 1$ and $i \geqslant 0$. Hence $M$ is Koszul.
\end{proof}

\begin{corollary}\label{minimal projecitve resolution of filtration module}
Keep the same assumptions for $A$ as in Theorem \ref{GrM is Koszul implies M is qK}. Suppose that $M$ is a left finite bounded below Koszul $A$-module.
\begin{enumerate}
\item If $\cdots\to P_i\to \cdots \to P_0\to M\to 0$ is a minimal graded projective resolution of ${}_AM$, then $$\cdots\to \Grj P_i[-i]\to\cdots\to \Grj P_0\to \Grj M\to 0$$ is a minimal graded projective resolution of $\Grj\! A$-module $\Grj\! M$.
\item $\pdim {}_AM=\pdim {}_{Gr_J A}\Grj M$, where  $\pdim$ means the projective dimension.
\end{enumerate}
\end{corollary}

\begin{corollary}\label{A is sqk iff Aop is}
Suppose that $A$ is a left and right finite $\mathbb{N}$-graded ring generated in degree $1$  with $A_0$ artinian. Then $A$ is a Koszul ring if and only if so is $A^{op}$.
\end{corollary}
\begin{proof}
By \cite[Proposition 2.2.1]{BGS1}, $\Grj A$ is a classical Koszul ring if and only if so is $(\Grj A)^{op}$. The conclusion follows from $\Grj A^{op}\cong (\Grj A)^{op}$, Theorem \ref{A sqk implies Gr A Koszul} and Theorem \ref{GrM is Koszul implies M is qK}.
\end{proof}

Corollary \ref{A is sqk iff Aop is} is a generalization of \cite[Corollary 4.3]{GM2} and \cite[Proposition 2.2.1]{BGS1}.

Combining the results of this section and the previous section, we have the following theorem, which is one of the main results in this paper.
\begin{theorem} \label{main-resut}
Suppose that $A$ is a left finite $\mathbb{N}$-graded ring generated in degree $1$  with $A_0$ artinian. Then the following are equivalent.
\begin{itemize}
\item[(1)] $A$ is a Koszul ring.
\item[(2)] $E(A)$ is a classical Koszul ring.
\item[(3)] $\Grj A$ is a classical Koszul ring.
\item[(4)] $E(A)$ is generated in degree $1$, and  as graded rings
$E(E(A)\cong \Grj A.$
\end{itemize}
Moreover if $A$ is right finite then the above statements are also equivalent to
\begin{itemize}
\item[(5)] $A^{op}$ is a  Koszul ring.
\end{itemize}

If $A$ is a Koszul ring, then for any left finite bounded below $A$-module $M$, the following are equivalent.
\begin{itemize}
\item[(1)] $M$ is a Koszul $A$-module.
\item[(2)] $E(M)$ is a classical Koszul $E(A)$-module.
\item[(3)] $\Grj M$ is a classical Koszul $\Grj A$-module.
\end{itemize}
If $M$ is finitely generated, then they are also equivalent to
\begin{itemize}
\item[(4)] $\mathcal{E}(M)$ is generated in degree $0$ and $\mathcal{F}(\mathcal{E}(M))\cong \Grj M$ as graded $\Grj A$-modules.
\end{itemize}
\end{theorem}

This theorem is a generalization of \cite[Theorems 2.5, 4.3]{MZ}.

Now we can give an example of quasi-Koszul ring but not Koszul.
\begin{example} \label{qk-but-not-sqk}
Let $R=k[x, y]/(x^2 + y^3, xy)$, which is a noetherian local ring with
$J=J(R)=(\bar{x}, \bar{y})$. A computation
due to Sj?din \cite[Theorem 5]{Sj} shows
$\Ext_R^\bullet(R/J, R/J) = T\langle u, v \rangle/(v^2, u^2v-vu^2)$, where $T\langle u, v \rangle$ is the free algebra in variables $u, v$ over $R/J$, and $|u| = 1 = |v|$.
 In particular $\Ext_R^\bullet(R/J, R/J)$ is generated
by $\Ext_R^1(R/J, R/J)$. Hence $R$ is quasi-Koszul. However, $R$ is not a Koszul ring by Theorem \ref{main-resut} since the associated graded algebra is
$k[x, y]/(x^2, xy, y^4)$, which is not classically Koszul.
\end{example}

\subsection{Questions in [CPS]}\label{questions-in-CPS}
In the representation theory of finite dimensional algebras, $\Grj A$, the associated graded algebra of $A$ with respect to the $J$-adic filtration, plays an important role. It reflects the homological property of $E(A)=\Ext^\bullet(A/J, A/J)$, which is often called the homological dual of $A$. For quasi-hereditary algebras $A$ appeared in Kazhdan-Lusztig theories, we often have $\Grj A \cong E(E(A))$ \cite[Theorem 2.2.1]{CPS}.
 In a series of papers, Cline, Parshall and Scott have been searching for Koszul structures in the quasi-hereditary algebras of interest in modular representation theory of algebraic groups.  The
general question they considered are what good properties of a finite dimensional algebra $A$
imply that the graded algebra $\Grj A$ is classically Koszul.

 The following two questions were proposed in \cite[Section 3]{CPS}.
\begin{itemize}
\item[(1)] Can it be determined if $A$ is classically Koszul entirely from knowledge of $E(A)$?
\item[(2)] Does the classical Koszul property for $\Grj A$ imply the same property for $A$?
\end{itemize}

A counterexample was given to show both of them have negative answers. Here is their example \cite[Example 3.2]{CPS}.

Let $B$ be the basic algebra with quiver
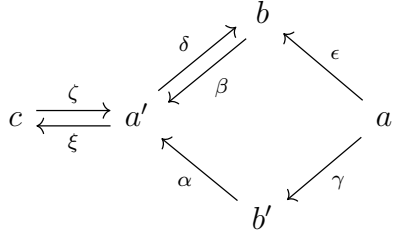
\begin{center}
\begin{tikzcd}
                                 &                                                                  & b \arrow[ld, "\beta", shift left] &                                                \\
c \arrow[r, "\zeta", shift left] & a' \arrow[ru, "\delta", shift left] \arrow[l, "\xi", shift left] &                                   & a \arrow[ld, "\gamma"] \arrow[lu, "\epsilon"'] \\
                                 &                                                                  & b' \arrow[lu, "\alpha"]           &
\end{tikzcd}
\end{center}
and defining relations
$$\xi\alpha=\xi\zeta=\delta\zeta=\alpha\gamma-\zeta\xi\beta\epsilon=\beta\delta-\zeta\xi=0.$$

Let $J_B$ be the ideal generated by all arrows. By setting the suitable weights of arrows, $B$ can be viewed as a locally finite graded algebra generated by $B_1$ over $B_0$. For example, let $\alpha$ and $\epsilon$ be of weight one and all others be of weight zero. Then,  $J_B$ is the graded Jacobson radical of $B$ in any case.

It was shown in \cite{CPS} that
$$\Grj B\cong 
E(E(B)),E(\Grj B) = \gExt_{\Grj B}^\bullet(B/J_B,B/J_B)$$
and they are classically Koszul, but $B \ncong \Grj B$. Thus $B$ is not a classical Koszul algebra.
In fact, it follows from Theorem \ref{main-resut} that $B$ is Koszul in our sense.


In general, for any finite dimensional algebra $A$, it follows from Theorem \ref{main-resut} that
\begin{itemize}
\item[(1)] if $\Grj A$ is classically Koszul,  then $A$ is Koszul viewed as a graded ring concentrated on degree $0$ in our sense.
\item[(2)] $A$ is Koszul in our sense if and only if that $E(A)$ is generated in degree $1$ and  $ E(E(A)) \cong \Grj A$ as graded rings.
\end{itemize}

Thus, the two questions in \cite{CPS} are answered in some sense.

\section{Generalized AS regular Koszul algebras}
In this section we assume that $A$ is a locally finite $\mathbb{N}$-graded $k$-algebra generated in degree $1$, where $k$ is a field.
If we assume further that $A$ is a Koszul algebra, then it is proved that $A$ is generalized AS regular if and only if so is $\Grj A$ (Theorem \ref{A is generalized AS regular iff GrA is}).
It is well known that, for any connected graded classical Koszul algebra $A$ with finite global dimension, $A$ is an AS regular algebra if and only if $E(A)=\gExt^\bullet_A(k,k)$ is a Frobenius algebra (see \cite[Proposition 5.10]{Sm}). This result was generalized to graded quiver algebra in \cite{MV1} (see also \cite{MV2}). The proof in \cite{MV1} works for classical Koszul algebras as stated in Proposition \ref{Classical case about the relation of Yoneda and generalized AS regular}.
We prove that this result holds for locally finite $\mathbb{N}$-graded Koszul algebras of finite global dimension by combining Theorem \ref{A is generalized AS regular iff GrA is} and Proposition \ref{Classical case about the relation of Yoneda and generalized AS regular}.

\subsection{Generalized AS regular algebras}
Here is the definition of generalized AS regular algebras \cite[Definition 1.4 and Theoren 5.4]{RR1} (see also  \cite[Definition 3.15]{MM} and \cite{MV2, MS}).

\begin{definition}\label{definition of GAS regluar}
A locally finite $\mathbb{N}$-graded algebra $A$ is called {\it generalized Artin-Schelter regular} (for short, AS regular) of dimension $d$ if the following conditions hold.
\begin{itemize}
\item[(1)] $A$ has global dimension $d < \infty$.
\item[(2)] For every graded simple $A$-module $M$, $\gExt^i_A(M,A)=0$ if $i\neq d$.
\item[(3)] $\gExt^d_A(-,A)$ induces a bijection between the isomorphism classes of graded simple $A$-modules and graded simple $A^{op}$-modules.
\end{itemize}
\end{definition}

In fact, by \cite[Theorem 1.5]{RR1}, $A$ is twisted Calabi-Yau (or equivalently, $A$ has graded Van den Bergh duality) if and only if $A$ is generalized AS-regular and $S=A/J$ is a separated $k$-algebra.

For any graded ring $A$, let $A$-$\gr$ be the category of finitely generated graded $A$-modules.

\begin{theorem}\cite[Theorem 5.2]{RR1}\label{equivalent definitions of AS regular}
Suppose that $A$ is a locally finite $\mathbb{N}$-graded algebra of finite global dimension $d$. If, for every graded simple $A$-module $M$, $\gExt^i_A(M,A)=0$ for $i\neq d$, then the following are equivalent:
\begin{itemize}
\item[(1)] $\gExt^d_A(-,A)$ gives a contravariant equivalence from $A_0$-$\gr$ to $A_0^{op}$-$\gr$.
\item[(2)] $\gExt^d_A(-,A)$ gives a contravariant equivalence from $S$-$\gr$ to $S^{op}$-$\gr$.
\item[(3)] $\gExt^d_A(S,A)\cong V$ as right $S$-modules, for some graded invertible $S$-bimodule $V$.
\item[(4)] $\gExt^d_A(S,A)\cong V$ as $S$-bimodules, for some graded invertible $S$-bimodule $V$.
\item[(5)] $\gExt^d_A(A_0,A)\cong \gHom_k(A_0, k)\otimes_{A_0} W$ as right $A_0$-modules, for some graded invertible $A_0$-bimodule $W$.
\item[(6)] $\gExt^d_A(A_0,A)\cong \gHom_k(A_0, k)\otimes_{A_0} W$ as $A_0$-bimodules, for some graded invertible $A_0$-bimodule $W$.
\item[(7)] $A$ is a generalized AS regular algebra.
\end{itemize}
\end{theorem}

\subsection{Characterization of generalized AS regular Koszul algebras}

Suppose that $M$ and $N$ are two graded $A$-modules, endowed with the $J$-adic filtration. Then $\gHom_A(M,N)$ is a filtered abelian group endowed with the induced filtration on Hom.  In fact, in this case,
$$F_n\gHom_A(M,N) = \{f \in \gHom_A(M,N) \mid f(M) \in J^nN\}.$$
This filtration is always exhaustive. It follows from the degree-wise nilpotency of $J$ that the filtration is separated if $N$ is bounded below.
There is a natural map
$$\varphi:\Gr \gHom_A(M,N)\to \gHom_{\Grj A}(\Grj M,\Grj N)$$
given by  $\varphi(\bar{f})(\bar{x})=f(x)+J^{n+i+1}N$, where $f\in F_n\,\gHom_A(M,N)$ and  $x\in J^iM$.

If $f\in \gHom_A(M,N),g\in \gHom_A(N,K)$, then $\varphi(\overline{gf})=\varphi(\bar{g})\varphi(\bar{f})$.

\begin{lemma}\label{varphi is isomorphism}
In general, $\varphi$ is injective. If  $M$ is a graded projective $A$-module, then $\varphi$ is an isomorphism.
\end{lemma}
\begin{proof}
It follows from the graded projetivity of ${}_AM$ that $M$ is $J$-adic filtered projective. The conclusion follows from \cite[Lemma I.6.9]{LO}.
\end{proof}

\begin{theorem}\label{A is generalized AS regular iff GrA is}
Suppose that $A$ is a locally finite Koszul algebra. Then $A$ is a generalized AS regular algebra of dimension $d$ if and only if $\Grj A$ is a generalized AS regular algebra of dimension $d$.
\end{theorem}
\begin{proof}
By Theorem \ref{A sqk implies Gr A Koszul}, $\Grj A$ is a classical Koszul algebra. It follows from Corollary \ref{minimal projecitve resolution of filtration module}  that $\gldim A=\pdim {}_AS=\pdim {}_{\Grj A}\Grj S=\gldim \Grj A$.

Let $P_\bullet\to S\to 0$ be a minimal graded projective resolution of ${}_AS$.
By Proposition \ref{minimal is fg}, each $P_i$ is a finitely generated $A$-module.

Now suppose that $A$ is generalized AS regular of dimension $d$.
 Let $(-)^*$ be the functor $\gHom_A(-,A)$. Then
$$0\to P_0^*\xrightarrow{d_1^*} P_1^*\to \cdots \xrightarrow{d_d^*} P_d^*\to \gExt_A^d(S,A)\to 0$$
is a minimal graded projective resolution of $\gExt^d_A(S,A)$ as an $A^{op}$-module.

By definition $F_nP_i^*=\{f \in P_i^* \mid f(P_i)\subseteq J^n\}$. In fact, it is the same as the $J$-adic filtration on the right $A$-module $P_i^*$, that is, $F_nP_i^*=P_i^*\, J^n$.
If $f \in P_i^*J^n$, then $f=\sum g_jx_j$ for some $g_j\in P_i^*$ and $x_j\in J^n$, and $f(P_i)=(\sum g_jx_j)(P_i)=\sum g_j(P_i)x_j\subseteq J^n$.
On the other hand, suppose $f \in P_i^*$ such that $f(P_i)\subseteq J^n$. Then $f=\sum_{\alpha} x_{\alpha}^* f(x_{\alpha}) \in P_i^*\, J^n$, where $\{x_{\alpha}, x_{\alpha}^*\}$ is a dual basis of the finitely generated graded projective module $P_i$.


Since $A$ is a locally finite Koszul algebra, it follows from Corollary \ref{A is sqk iff Aop is} that $A^{op}$ is a Koszul algebra. Thus $\gExt_A^d(S,A)$ is a Koszul $A^{op}$-module. By Corollary \ref{minimal projecitve resolution of filtration module},
$$0\to \Grj P_0^*[-d]\to \cdots \to\Grj P_d^*\to \Grj\gExt_A^d(S,A)\to 0$$
is a minimal graded projective resolution of $\Grj\gExt_A^d(S,A)$ as a $(\Gr_J A)^{op}$-module.
It follows from Lemma \ref{varphi is isomorphism} that
\begin{align*}
\Grj P_i^*&=\oplus (P_i^*J^n/P_i^*J^{n+1}) =\oplus (F_nP_i^*/F_{n+1}P_i^*)\\
          &= \Gr_J (\gHom_A(P_i, A)) \cong \gHom_{\Grj A}(\Grj P_i,\Grj A),
\end{align*}
and the following is an isomorphism of complexes,
\begin{center}
\begin{tikzcd}[column sep=tiny]
0 \arrow[r] & \Grj P_0^*[-d] \arrow[r] \arrow[d] & \Grj P_1^*[-d+1] \arrow[r] \arrow[d] & \cdots \arrow[r] & \Grj P_d^* \arrow[r] \arrow[d]  & 0  \\
0 \arrow[r]          & (\Grj P_0)^*[-d] \arrow[r]         & (\Grj P_1)^*[-d+1] \arrow[r]         & \cdots \arrow[r] & (\Grj P_d)^* \arrow[r]           & 0 {}
\end{tikzcd}
\end{center}
where  $(\Grj P_i)^*=\gHom_{\Grj A}(\Grj P_i,\Grj A)$.

Again by Corollary \ref{minimal projecitve resolution of filtration module} and $S=\Grj S$,
$$\gExt_{\Grj A}^i(S,\Grj A)\left\{
\begin{aligned}
&=0,&i\neq d\\
&\cong \Grj\gExt_A^d(S,A)[d], &i=d
\end{aligned}
\right.
$$
 as right $\Grj A$-modules.
By Theorem \ref{equivalent definitions of AS regular}, $J\gExt_A^d(S,A)=\gExt_A^d(S,A)J$ and $\Grj\gExt_A^d(S,A)\cong \gExt_A^d(S,A)$ is invertible as $S$-bimodule. By Theorem \ref{equivalent definitions of AS regular} again, $\Grj A$ is a generalized AS regular algebra.

Conversely, suppose  $\Grj A$ is a generalized AS regular algebra of dimension $d$.
Since $\Grj A$ is a classical Koszul algebra, $S$ as a $\Grj A$-module  has a linear projective resolution:
$$0\to Q_d\to Q_{d-1}\to \cdots \to Q_0\to S\to 0.$$
Then
$$0\to Q_0^*\to \cdots \to Q_d^*\to \gExt_{\Grj A}^d(S,\Grj A)\to 0$$
is a projective resolution of $\gExt_{\Grj A}^d(S,\Grj A)$. Since $Q_d^*$ is generated in degree $-d$ and $\gExt_{\Grj A}^d(S,\Grj A)$ is semisimple, $\gExt_{\Grj A}^d(S,\Grj A)$ is concentrated in degree $-d$. Therefore,
$\gExt_{\Grj A}^d(S,\Grj A)_i=0,\forall i\neq -d$ and $\gExt_{\Grj A}^d(S,\Grj A)_{-d}$ is an invertible $S$-bimodule.

For convenience, let us denote the complex $\gHom_A(P_{\bullet}, A)$ by $C^\bullet$:
$$0\to \gHom_A(P_0,A)\xrightarrow{d_1^*} \gHom_A(P_1,A)\xrightarrow{d_2^*}\cdots \xrightarrow{d_d^*} \gHom_A(P_d,A)\to 0.$$
Consider the $i$-th shift of the $J$-adic filtration on $P_i^*$. For convenience, we still use $F_n$ to denote the new filtration, that is,
$$F_n\gHom_A(P_i,A):=\{f \mid f(P_i)\subseteq J^{n+i}\}.$$
For any $f\in F_n\gHom_A(P_i,A)$,
$$d_{i+1}^*(f)(P_{i+1})=f(d_{i+1}(P_{i+1}))=f(\Ker d_i)\subseteq f(JP_i)\subseteq J^{n+i+1},$$
which means $d_{i+1}^*(f)\in F_n\gHom_A(P_{i+1},A)$. Therefore, $C^\bullet$ is a filtered cochain complex.

Now we look into the spectral sequence of the filtered complex $FC^\bullet$. We will use the same terminology and notation as \cite[Section 5.4, 5.5]{W}, but  the cohomological version.

The objects in $0$-page are
\begin{align*}
E_0^{pq}&=\dfrac{F_p\gHom_A(P_{p+q},A)}{F_{p+1}\gHom_A(P_{p+q},A)}\\
&\cong \big(\Gr \gHom_A(P_{p+q},A)[p+q]\big)_p\\
&\cong  \big(\gHom_{\Grj A}(\Grj P_{p+q},\Grj A)[p+q]\big)_p\\
&=\gHom_{\Grj A}((\Grj P_{p+q})[-p-q],\Grj A)_p.
\end{align*}
The objects  in $1$-page are
$$E_1^{pq}=\gExt_{\Grj A}^{p+q}(S,\Grj A)_p\left\{
\begin{aligned}
&=0,&(p,q)\neq (-d,2d)\\
&\cong \gExt_{\Grj A}^d(S,\Grj A)_{-d}, &(p,q)= (-d,2d)
\end{aligned}
\right.
$$ as right $\Grj A$-modules.
So, the spectral sequence $\{E_n^{pq}\}$ is bounded, and
$$E_\infty^{pq}\left\{
\begin{aligned}
&=0,&(p,q)\neq (-d,2d)\\
&\cong \gExt_{\Grj A}^d(S,\Grj A)_{-d}, &(p,q)= (-d,2d)
\end{aligned}
\right.
$$ as right $\Grj A$-modules.

Although the filtration is not bounded below, and the classical convergence theorem \cite[Theorem 5.5.1]{W} does not apply directly, this spectral sequence is still convergent to the cohomological groups of $C^\bullet$, which we check in the following.

Let us recall the notations first. Let $\eta_p:F_pC^{p+q}\to F_pC^{p+q}/F_{p+1}C^{p+q}$ be the natural projection and $\partial$ be the differentials of $C^\bullet$.
\begin{align*}
&A_r^{pq}=\{f\in F_pC^{p+q}\mid \partial(f)\in F_{p+r}C^{p+q+1}\}\\
&Z_r^{pq}=\eta_p(A_r^{pq})=\dfrac{A_r^{pq}+F_{p+1}C^{p+q}}{F_{p+1}C^{p+q}}\\
&B_r^{pq}=\eta_p(\partial(A_{r-1}^{p-r+1,q+r-2}))=\dfrac{\partial(A_{r-1}^{p-r+1,q+r-2})+F_{p+1}C^{p+q}}{F_{p+1}C^{p+q}}\\
&Z_\infty^{pq}=\cap_rZ_r^{pq}=\dfrac{\cap_r(A_r^{pq}+F_{p+1}C^{p+q})}{F_{p+1}C^{p+q}}\\
&B_\infty^{pq}=\cup_r B_r^{pq}=\dfrac{\partial(C^{p+q-1})\cap F_pC^{p+q}+F_{p+1}C^{p+q}}{F_{p+1}C^{p+q}} \\
& E_\infty^{pq}=Z_\infty^{pq}/B_\infty^{pq}
\end{align*}

Let $z_\infty^{pq}=\dfrac{\Ker \partial\cap F_pC^{p+q}+F_{p+1}C^{p+q}}{F_{p+1}C^{p+q}	}\subseteq Z_\infty^{pq}$ and $e_\infty^{pq}=z_\infty^{pq}/B_\infty^{pq}$.
In general $e_\infty^{pq} \subseteq E_\infty^{pq}$.

The cohomological groups $\{H^\bullet\}$ of $C^\bullet$ have the standard filtration induced by $FC^\bullet$ given by
$$F_pH^{p+q}:=\dfrac{\Ker\partial\cap F_pC^{p+q}+\partial(C^{p+q-1})}{\partial(C^{p+q-1})}.$$
Then
\begin{align*}
\dfrac{F_pH^{p+q}}{F_{p+1}H^{p+q}}&\cong\dfrac{\Ker\partial\cap F_pC^{p+q}+\partial(C^{p+q-1})}{\Ker\partial\cap F_{p+1}C^{p+q}+\partial(C^{p+q-1})}\\
&\cong \dfrac{\Ker\partial\cap F_pC^{p+q}}{\Ker\partial\cap F_{p+1}C^{p+q}+\partial(C^{p+q-1})\cap F_pC^{p+q}}.
\end{align*}
Hence
\begin{align*}
 e_\infty^{pq}&\cong \dfrac{\Ker \partial\cap F_pC^{p+q}+F_{p+1}C^{p+q}}{\partial(C^{p+q-1})\cap F_pC^{p+q}+F_{p+1}C^{p+q}}\\
&\cong \dfrac{\Ker\partial\cap F_pC^{p+q}}{\Ker\partial\cap F_{p+1}C^{p+q}+\partial(C^{p+q-1})\cap F_pC^{p+q}}\\
&\cong \dfrac{F_pH^{p+q}}{F_{p+1}H^{p+q}}.
\end{align*}

Clearly, $\cap_rA_r^{pq}+F_{p+1}C^{p+q}\subseteq\cap_r(A_r^{pq}+F_{p+1}C^{p+q})$.

Now, in our case, $F_pC^n=C^nJ^{p+n}$ is a bounded below graded $A^{op}$-module. So, $A_r^{pq}=\partial^{-1}(F_{p+r}C^{p+q+1})\cap F_pC^{p+q}$ is also a bounded below graded $A^{op}$-module.
For a fixed degree $i$, by Lemma \ref{locally nilpotent}, there is an integer $r_i$ such that for all $r\geqslant r_i$, $(F_{p+r}C^{p+q+1})_i=0$. Thus $(A_r^{pq})_i=(\Ker \partial)_i\cap (F_pC^{p+q})_i$  for all $r\geqslant r_i$. Therefore
\begin{align*}
&\mathop{\cap}\limits_{r}(A_r^{pq}+F_{p+1}C^{p+q})_i=\mathop{\cap}\limits_{r}((A_r^{pq})_i+(F_{p+1}C^{p+q})_i)\\
=&\mathop{\cap}\limits_{r=0}^{r_i}((A_r^{pq})_i+(F_{p+1}C^{p+q})_i=(A_{r_i}^{pq})_i+(F_{p+1}C^{p+q})_i \\
=&(\mathop{\cap}\limits_{r} A_r^{pq})_i  + (F_{p+1}C^{p+q})_i=(\mathop{\cap}\limits_{r}A_r^{pq}+F_{p+1}C^{p+q})_i.
\end{align*}
 Then $\mathop{\cap}\limits_{r}(A_r^{pq}+F_{p+1}C^{p+q})=\mathop{\cap}\limits_{r}A_r^{pq}+F_{p+1}C^{p+q}$, as $i$ is arbitrary. So
\begin{displaymath}
Z_\infty^{pq}=\dfrac{\cap_rA_r^{pq}+F_{p+1}C^{p+q}}{F_{p+1}C^{p+q}}.
\end{displaymath}

 If $f\in \cap_rA_r^{pq}$, then for any $r$,
$$\partial(f)=f\circ d_{p+q+1}\in F_{p+r}\gHom_A(P_{p+q+1},A)$$ and
$$f\circ d_{p+q+1}(P_{p+q+1})\subseteq J^{p+r+p+q+1}.$$
By the arbitrariness of $r$ and Lemma  \ref{locally nilpotent}, $f\circ d_{p+q+1}(P_{p+q+1})\subseteq\cap_i J^i = 0$. Thus $f\in\Ker \partial$ and $\cap_rA_r^{pq}=\Ker \partial\cap F_pC^{p+q}$.

It follows that $Z_\infty^{pq}=\dfrac{\Ker \partial\cap F_pC^{p+q} + F_{p+1}C^{p+q}}{F_{p+1}C^{p+q}}$ and  $e_\infty^{pq}=E_\infty^{pq}$.

By definition, $F_{-n}C^n=C^n$, and so $F_{-n}H^n=\dfrac{\Ker d_{n+1}^*}{\im d_{n}^*}=\gExt^{n}_A(S,A)$ for any $n$,

If $n\neq d$, 
then $F_iH^n/{F_{i+1}H^n} \cong E^{i, n-i}_{\infty}=0$.
Hence $$F_{-n}H^{n}=F_{-n+1}H^{n}=\cdots=F_iH^{n}=\dfrac{\Ker d_{n+1}^*\cap F_iC^{n}+\im d_{n}^*}{\im d_{n}^*}=\cdots$$
for all $i > -n$. It follows that, for any $i>-n$,
$$\Ker d_{n+1}^*=\Ker d_{n+1}^*\cap F_iC^{n}+\im d_{n}^*.$$
By degree-wise nilpotency again,
$\Ker d_{n+1}^*=\im d_{n}^*.$
Therefore, $H^n=0$, that is,  $\gExt_A^n(S,A)=0$ for all $n\neq d$.

If $n=d$, then $F_{-d+1}H^d=F_iH^d$ for all $i>-d$ as only $E_\infty^{-d,2d}\neq 0$. It follows that, for all $i>-d$,
$$\Ker d^*_{d+1}\cap F_{-d+1}C^d+\im d_d^*=\Ker d^*_{d+1}\cap F_i C^d+\im d_d^*.$$
Similarly, by the degree-wise nilpotency,
$$\Ker d^*_{d+1}\cap F_{-d+1}\gHom_A(P_d,A)+\im d_d^*=\im d_d^*.$$
Therefore $F_{-d+1}H^d=F_{-d+2}H^d=\cdots=0.$
Hence
$$F_{-d}H^d/F_{-d+1}H^d=H^d=\gExt^d_A(S,A)\cong E^{-d,2d}_\infty\cong\gExt_{\Grj A}^d(S,\Grj A)_{-d}$$
as right $S$-modules. Therefore, as right $A$-modules,
$$\gExt_A^d(S,A)\cong \gExt_{\Grj A}^d(S,\Grj A)_{-d}.$$
Since $\gExt_{\Grj A}^d(S,\Grj A)_{-d}$ is an invertible $S$-bimodule,
it follows from Theorem \ref{equivalent definitions of AS regular} that $A$ is a generalized AS regular algebra.
\end{proof}

The following proposition studies the relations between the generalized AS regular property of a graded classical Koszul quiver algebra and its Yoneda Ext algebra, see \cite[Theorem 5.1]{MV1} or \cite[Proposition 4.1]{MV2}. 


\begin{proposition}\label{Classical case about the relation of Yoneda and generalized AS regular}
Suppose that $A$ is a graded classical Koszul quiver algebra. Let $E(A)=\gExt_A^\bullet(S,S)$ be the Yoneda Ext algebra of $A$. Then for any graded simple $A$-module $N$ with $\pdim N=d$, the following are equivalent.
\begin{itemize}
\item[(1)] $N$ satisfies
$$\gExt_A^i(N,A)=\left\{
\begin{aligned}
&0,&i\neq d\\
&N', &i=d
\end{aligned}
\right.
$$
where $N'$ is a graded simple right  $A$-module.
\item[(2)] $\gExt_A^\bullet(N,S)$ is an injective $E(A)$-module.
\end{itemize}
\end{proposition}

Now we are ready to give the main result in this section. Recall that an $\mathbb{N}$-graded algebra $A$ is called basic if the degree $0$ part of $A/J$ is a finite direct sum of $k$.

\begin{theorem}\label{Char-generalized Koszul AS-regular algebra}
Suppose that $A$ is a basic locally finite Koszul algebra of global dimension $d$. Let $E(A)=\gExt_A^\bullet(S,S)$ be the Yoneda Ext algebra of $A$. Then the following are equivalent.
\begin{itemize}
\item[(1)] $A$ is a generalized AS regular algebra of dimension $d$.
\item[(2)] $E(A)$ is a self-injective algebra.
\end{itemize}
\end{theorem}
\begin{proof}
It follows from Theorem \ref{A is generalized AS regular iff GrA is}, Proposition \ref{Classical case about the relation of Yoneda and generalized AS regular} and Theorem \ref{A sqk implies Gr A Koszul}.
\end{proof}

\section*{Acknowledgements} This research is partially supported by  the National Key Research and Development Program of China (Grant No. 2020YFA0713200) and the National Science Foundation of China (Grant No. 11771085).

\end{document}